\documentclass{article}

\usepackage{amsmath, amssymb}
\usepackage{amsthm}	
\usepackage[normalem]{ulem}
\usepackage{graphicx}

\usepackage[utf8]{inputenc}		

\usepackage{enumitem}

\usepackage{xcolor} 

\usepackage{hyperref}

\newcommand{\clv}{\color{violet}}
 \title{\textbf{Directed polymers on infinite graphs}}

\newcommand{\IP}{{\mathbb P}}
\newcommand{\IE}{{\mathbb E}}
\newcommand{\DP}{{\mathrm P}}
\newcommand{\DE}{{\mathrm E}}
\newcommand{\IEtilde}{\widetilde{\IE}}
\newcommand{\IPtilde}{\widetilde{\IP}}
\renewcommand{\b}{\beta}

\newcommand*\Z{\mathbb{Z}}

\newcommand{\cvlaw}{\stackrel{{ (d)}}{\longrightarrow}}
\newcommand{\eqlaw}{\stackrel{(d)}{=}}

\newcommand*\cvIP{\overset{\IP}{\longrightarrow}}

\newcommand{\dd}{\mathrm{d}}

\newtheorem{theorem}{Theorem}[section]
\newtheorem{lemma}[theorem]{Lemma}
\newtheorem{proposition}[theorem]{Proposition}
\newtheorem{corollary}[theorem]{Corollary}

\newtheorem*{theorem*}{Theorem}
\newtheorem*{lemme*}{Lemme}

\newtheorem*{proposition*}{Proposition}

\newtheorem{assumption}[theorem]{Assumption}
\newtheorem{remark}[theorem]{Remark}

\theoremstyle{definition}
\newtheorem{definition}[theorem]{Definition}

\newtheorem{question}{Question}[section]

\AtBeginDocument{
   \def\MR#1{}  }

    \newcommand{\CC}{\textcolor{blue}}
     \newcommand{\IS}{\textcolor{violet}}
\author{Cl\'ement Cosco, Inbar Seroussi, Ofer Zeitouni
}
\date{}

\begin{document}
\maketitle
\begin{abstract}
  We study the directed polymer model for general graphs (beyond $\mathbb Z^d$)
  and random walks.
  We provide sufficient conditions for the existence or
  non-existence of
  a weak disorder phase, of an $L^2$ region, and
   of very strong disorder, in terms of properties of
  the graph and of the random walk. We study in some  detail
  (biased) random walk on various trees including the
  Galton Watson trees, and provide  a range of other examples
  that illustrate counter-examples to intuitive extensions of the
  $\mathbb Z^d$/SRW result.
\end{abstract}

\section{Introduction}
The model of polymers in random environment, that is of random walk
that is weighted by a random time-space field, has a long history
in statistical physics, both on its own right and as a tool in understanding
interfaces, see \cite{HuHe85} for an early occurence. It soon appeared also
in the mathematical literature, see \cite{ImbrieSpencer88,Bolth89}. We
refer to
\cite{CStFlour} for a recent overview of
the subject from a mathematical perspective,
and a concise history. Most of the mathematical work has
focused on the model where
the walk associated to the polymer is a symmetric random walk
on the lattice $\mathbb Z^d$or on approximations of the walk on the lattice, such as downward paths on trees \cite{DerridaSpohn88,EBuffet1993Dpot}, diffusions on the $d$-dimensional discrete torus \cite{EckmannJ.1989TlLe} or on the cylinder \cite{BrunetE2000Pdot}, or simple random walk on the complete graph \cite{CoFlRa19}. In either case,
the study of
the polymer
is closely related,
via the Feynman-Kac representation,
to the study of
a stochastic heat equation (SHE) on the underlying lattice/tree.

Recently,  as part of a  study of
stochastic dynamics equivalent to the SHE on large 
networks,
Sochen and the second author
\cite{seroussi2018spectral}
discussed the effect of the underlying network topology on the dynamics.
Using dynamic field theory, the multiplicative noise can
be translated to an interaction term
between the eigen-functions of the graph Laplacian.
The second moment of the solution is then calculated using expansion
in these eigen-functions. Similar to the lattice topology,
for transitive graphs the different phases of the model (defined below)
depend on the spectral dimension of the graph.

Motivated by that work, we study in this paper  how key notions that have been
developed in the study of directed polymers on $\mathbb Z^d$
translate to the situation where the associated walk is defined on
various infinite graphs.
Of  particular interest is the case where
the underlying graph is itself random (such as various percolation models),
or at least irregular, and the relation between the
transience or recurrence of the random walk on the graph
and the phase transitions among different regimes.
As we will see, new phenomena emerge, and the structure of the
underlying graph
has an important effect on the behavior of the polymer.
Naturally, we emphasize these aspects of the theory.
Our goal in this paper is to initiate the study of
these interesting models and raise new questions, rather
than providing complete answers to all models. The conclusion and open
problem section \ref{sec-open} states several open questions that we
find of interest.

We mention two other papers that adopt a similar point of view. Polymers for which the underlying walk is a general Markov chain have been studied in \cite{CaGuHuMe04}, in the specific case where the chain is positive recurrent. For the related model of the parabolic Anderson model (PAM) (which studies the SHE equation when the noise only depends on time), the recent  \cite{PAMonGW_HoKoSa20} focuses on the PAM on Galton-Watson trees and locally tree-like structures such as the configuration model.

In the rest of this introduction, we
explicitly introduce the polymer model, define the different phases,
and state some general theorems concerning the existence and properties of various phases. These are easy extensions of the standard results for the case
of simple random walk on the lattice. We also introduce certain graphs that
will be a good source of counter-examples. Our main results are stated in
Section \ref{sec-2}. Section \ref{sec-spec} introduces three
classes of graphs with associated Markov chains,
that are used to illustrate various features and are interesting on their own rights. Those are the lattice super-critical percolation cluster, the
biased walk on Galton-Watson trees, and the canopy graph.
The proofs of all statements appear in Sections \ref{sec-3}--\ref{sec-5}.
Section \ref{sec-open} contains concluding remarks and the statement of several open problems.
\subsection{The polymer model}
 To
set the stage for a description of our results, we begin by introducing our
model of random polymer.
Let $G=(V,E)$ denote a connected (undirected)
graph with (infinite) vertex set $V$ and
set of edges $E\subset V\times V$. We let $d(x,y)$ denote the graph distance
between $x,y\in V$, i.e. the legth of the shortest path connecting $x,y$.

Associated with the graph is a nearest-neighbor discrete time
Markov chain $S=(S_k)_{k\geq 0}$
with (time-homogeneous) transition matrix
$P(x,y), x,y\in V$, where $P(x,y)=0$ if $(x,y)\not\in E$.
We denote by $\DP_x$ the law of $(S_k)_{k\geq 0}$ where $S_0=x$.
The expectation under $\DP_x$ is denoted by $\DE_x$
and we set $p_n(x,y) = \DP_{x}(S_n=y)$.
We remark that often,
we consider the simple random walk (SRW) case
determined by  $P(x,y)=1/d_x$ when $(x,y)\in E$, with $d_x$ the degree of
$x\in V$. This of course is only defined when the degree
is locally finite, i.e. so that $d_x<\infty$ for all $x\in V$. Throughout, we write $S,S'$ for two independent copies of $S$.

The third component in the definition of the polymer is the
\emph{environment}, which
is a set of i.i.d.\ random variables $\omega(i,x)$
with $i\in \mathbb N$ and $x\in V$. For concreteness, we
chose the nomalization that makes $\omega(i,x)$ of
of mean zero and variance one. The law of the environment is denoted
$\IP$, with expectation denoted by $\IE$. We also let $\mathcal G_n$ denote the sigma-algebra generated by $\{\omega(i,x),i\leq n,x\in V\}$.

Throughout the paper, we make the following blanket assumption on the
random walk and on the environment.
\begin{assumption} \label{ass:1st}
\begin{enumerate}
\item The Markov chain $((S_k)_k,V)$ is irreducible
 and $(G,S)$ is locally finite, i.e.\ $\bar d_x:=
 \sum_{y: (x,y)\in E} {\bf 1}_{p(x,y)>0}<\infty$ for all $x\in V$;
\item  $\Lambda(\b): = \log \IE[e^{\b \omega(i,x)}]$ is finite for all $\b>0$. 
\end{enumerate}
\end{assumption}


Continuing with definitions, the \emph{polymer measure} $\DP^{n,\b}_x$ of horizon $n$ and \emph{inverse temperature} $\b\geq 0$ is the probability measure on the paths $S=(S_k)_{k\geq 0}$ given by
\begin{equation} \label{eq:defPolymerMeasure}
\dd \DP^{n,\b}_x(S) = \frac{e^{\b \sum_{i=1}^n \omega(i,S_i)}}{Z_n(x)} \dd \DP_x(S),
\end{equation}
where the \emph{partition function}  $Z_n(x)$ satisfies
\begin{equation}
 Z_n(x) = Z_n(x,\b,\omega) = \DE_x \left[ e^{\b \sum_{i=1}^n \omega(i,S_i)}\right].
\end{equation}
Under the polymer measure $\DP^{n,\b}_x$, the polymer path $(S_k)$ favors parts of the environment that take high values, and the parameter $\b$ tunes the intensity of this preference. One thus expects a transition between the
\textit{delocalized} (small $\b$) regime, where the
polymer does not exhibit a qualitative change of behavior
compared to the original walk, and the
\textit{localized} (large $\b$) regime, where
the polymer localizes in attractive parts of the environment.

\subsection{Weak and strong disorder and
their consequences} \label{subsec:weakStrongIntro}
An important quantity
in the study of the localized/delocalized   transition
is the \emph{normalized partition function}:
\begin{equation}
W_n(x) = W_n(x,\b,\omega) = Z_n(x)/\IE[Z_n(x)] = Z_n(x) e^{-n\Lambda(\b)}.
\end{equation}
It is straightforward to check that for fixed $x$,
$W_n(x)$ defines a mean-one, positive martingale
with respect to $(\mathcal G_n)_n$, which therefore
converges $\IP$-a.s.\ to a limit $W_\infty(x,\b)$.
The following easy 0-1 law
holds in our general context.
\begin{proposition}\label{prop:StrongWeakDisorderIntro} For all $\b\geq 0$,
\begin{equation*}
\begin{aligned}
 \text{either } & \ \forall x\in V,\   W_\infty(x,\b) > 0 \text{ a.s,} \quad \text{(weak disorder)}\\
 \text{or } & \ \forall x\in V,\   W_\infty(x,\b) = 0 \text{ a.s.} \quad \text{(strong disorder)}
\end{aligned}
\end{equation*}
Moreover, there is a critical parameter $\b_c\in[0,\infty]$ such that weak disorder  holds if $\b<\b_c$ and strong disorder
 holds if $\b>\b_c$.
 \end{proposition}
For the lattice/SRW  model, it is known that $\b_c = 0$ in dimensions $d=1,2$
and that $\b_c\in(0,\infty)$ when $d\geq 3$, see \cite{Lacoin10,CStFlour}.

One expects that the weak/strong disorder transition corresponds to the
localized/delocalized one.
Indeed, for
the lattice/SRW model,  \cite{CY06} show
that in the whole weak disorder region,
the polymer path satisfies a functional central limit theorem.
Their argument adapts to our general context with some restrictions,
as follows.
Let $|S_n|=d(S_n,S_0)$ and let $|S^{(N)}| = (|S_{Nt}| / \sqrt N)_{t\in [0,1]}$
denote the continuous time process obtained by interpolation.
\begin{theorem}[\cite{CY06}]
  \label{theo-FCLT}
  Let $S_0=x\in V$.
Assume that: \\
(i) $(|S_k|)$ satisfies an almost-sure central limit theorem in the sense that for all sequence $N_k$ such that $\inf_{k} N_{k+1}/N_k > 0$, for any bounded and Lipschitz function $F$ of the path, as $N\to\infty$,
\[\frac{1}{N} \sum_{k=1}^N F\left(\left|S^{(N_k)}\right|\right)  \to \DE\left[F\left(\left|B\right|\right)\right], \quad \DP_x \text{-a.s,}
\]
where $(B_t)$ is a centered, real-valued Brownian motion with $\DE[|B_1|^2] > 0$.\\
(ii) $(W_n(x,\b))_n$ is uniformly integrable.

Then, as $n\to\infty$,
\begin{equation}
\DE_{x}^{n,\b}\left[F\left(|S^{n}|\right)\right] \cvIP \DE\left[F(\left|B\right|)\right].
\end{equation}
\end{theorem}
\begin{remark}
  Condition (ii) of Theorem \ref{theo-FCLT}
  implies that weak disorder holds. The converse may not hold,
  see the discussion in Section \ref{sec:UIintro}.
\end{remark}
\begin{remark}
  Under weak disorder, one can consider the limit polymer measure, defined as
$\mu_{polymer} = \lim_{n\to\infty} \DP_x^{n,\b}$.
Proposition 4.1 in \cite{CY06} states that $\mu_{polymer}$ is well defined and
is absolutely continuous with respect to the original measure $\DP_x$; the
proof carries over to
our general framework.
\end{remark}

Let $I_n(x) = (\DP^{n-1,\b}_x)^{\otimes 2}(S_n={S}'_n)$ denote
the probability for two independent polymer paths (in the same
environment) to end at the same point.
As noted in \cite[Remark 2.5]{CSY03}, the next theorem holds under
the mere assumption that
$(S_k)$ is a Markov chain (in particular, irreducibility of $(S,P)$ is not necessary).
\begin{theorem}[\cite{CSY03,CaHu02}] \label{th:CSY}
  For all $\b > 0$, $x\in V$,
\begin{equation}
 \{W_\infty(x) = 0\} = \left\{\sum_{n\geq 0} I_n(x) = \infty\right\}, \quad \IP\text{-a.s.}
\end{equation}
\end{theorem}


In the other extreme,
strong
localization properties in the entire strong disorder region have been shown
for the lattice/SRW model \cite{CaHu02,CSY03}.

We close this subsection by noting that while $W_n(x)$ may be very different
from its expectation, this is not the case for $\log W_n(x)$.
The next theorem
was proved in the lattice case  in \cite{LiuWat}. An inspection of
the proof reveals that it
transfers directly to our setup.
\begin{theorem}[\cite{LiuWat}] \label{th:ConcentrInequality}
For all $\b>0$, there exists a finite constant $C=C(\b)$ (that is independent of the graph structure) such that for all $x\in G$,
\begin{equation} \label{eq:ConcentrInequality}
\IP\left(\left|\frac{\log W_n(x)}{n} -\frac {\IE[\log W_n(x)]}{n}\right| \geq \varepsilon\right) \leq \begin{cases}
2 e^{- n C \varepsilon^2} & \text{if } 0\leq \varepsilon \leq 1,\\
2 e^{- n C {\varepsilon}} & \text{if } \varepsilon \geq  1.
\end{cases}
\end{equation}
\end{theorem}


\subsection{Very strong disorder} \label{subsec:verystrongIntro}
Under strong disorder, we have by Proposition \ref{prop:StrongWeakDisorderIntro} that $W_n(x,\beta)\to_{n\to\infty} 0$, a.s. We say that \emph{very strong disorder} holds
if that decay is exponential, that is, if
\begin{equation} \label{eq:veryStrongDisordermain}
\forall x\in V,\ \bar{p}(x):= \limsup_{n\to\infty} \frac{1}{n} \IE \log W_n(x,\b) < 0.
\end{equation}
\begin{proposition} \label{prop:veryStrongDis}
  The limit $\bar{p}:=\bar{p}(x)$ does not depend on $x\in V$. Moreover,
there exists a critical parameter $\bar{\b}_c\in[0,\infty]$, such that very strong disorder holds for $\b > \bar \b_c$ and $\,\bar{p}=0$  when $\b<\bar \b_c $. Finally,
\begin{equation} \label{eq:quenchedVerystrong}
\IP\text{-a.s.},\quad \limsup_{n\to\infty} \frac{1}{n} \log W_n(x,\b) = \limsup_{n\to\infty} \frac{1}{n} \IE \log W_n(x,\b).
\end{equation}
\end{proposition}

From the definitions, we clearly have that $\b_c\leq \bar{\b}_c$.
For the lattice/SRW model, it is known that $\bar{\b}_c=0$ when $d=1,2$  (see \cite{Lacoin10,CV06}
), while $\bar \b_c \in (0,\infty)$ when $d\geq 3$. In this latter case, the question whether or not $\bar \b_c = \b_c$ is, to our knowledge, still open. Still on the lattice, there exists a random walk with heavy-tailed jumps such that
 $\b_c < \bar \b_c$, see
\cite{Viveros20}.

In our general framework, we have the following:
\begin{proposition}[Very strong disorder always holds for large $\b$] \label{prop:veryStrongDisAlwaysHolds} Assume that there exists $d<\infty$ so that
  $d_x\leq d$ for all $d$.
  Assume further
  that the support of the law of $\omega(i,x)$ is unbounded from above. Then, there exists $\b_0\geq 0$ such that very strong disorder holds for all $\b>\b_0$.
\end{proposition}

Very strong disorder implies the following strong localization property.
The proof, given originally for the lattice/SRW model, carries over
without change to our setup (see \cite[Remark 2.5]{CSY03}).
\begin{theorem}[\cite{CSY03}]\label{th:CSY2} 
 Very strong disorder holds if and only if there is some $c>0$ such that
\begin{equation}
\forall x\in V,\quad \liminf_{n\to\infty} \frac{1}{n} \sum_{k=1}^{n} I_k(x) \geq c,\quad \IP\text{-a.s.}
\end{equation}
or, equivalently, if there is some $c>0$ such that
\begin{equation} \label{eq:mass_concentration}
\forall x\in V,\quad \liminf_{n\to\infty} \frac 1n \sum_{k=1}^n \sup_{y\in V} \DP^{k,\b}_x(S_k=y) \geq c,\quad \IP\text{-a.s.}
\end{equation}
In particular, under strong disorder, there exists $c>0$ such that
\[\forall x\in V,\quad \limsup_{n\to\infty} \sup_{y\in V} \DP^{n,\b}_x(S_n=y) \geq c,\quad \IP\text{-a.s.}
\]
\end{theorem}

For the lattice/SRW model, the authors in \cite{BatesChatterjee} have gone deeper into the description of the (endpoint) localization phenomena. They showed that in the full very strong disorder region, the mass of the endpoint concentrates asymptotically on some small islands -- a phenomena also called \emph{asymptotic pure atomicity}. More precisely, they proved that if
 \[\mathcal A_k^\varepsilon = \{x\in \mathbb Z^d:\  \DP^{k,\b}_0(S_k=x)\},\]
then, for every sequence $(\varepsilon_k)_{k\geq 0}$ vanishing as $n\to\infty$,
\[\lim_{n\to\infty} \sum_{k=0}^{n-1} \DP^{k,\b}_0\left(S_k\in A_k^{\varepsilon_k}\right) = 1,\quad \IP\text{-a.s.}\]
 In a related continuous setting, localization properties of the Brownian
 polymer \emph{full path} (i.e.\ not restricted to the endpoint properties) have been shown to hold deep inside the very strong disorder region
 in \cite{CYBMPO2}.


\subsection{The $L^2$-region} \label{subsec:L2regionIntro}
A range of parameters
that plays an important role in the literature because
it is tailored to moment computations,
is the $L^2$-\emph{region}, which 
corresponds to the set of $\b$'s such that the martingale $(W_n(x,\b))_n$ is bounded in $L^2$, i.e.\
\begin{equation} \label{eq:L2conditionIntro}
 \forall x\in V,\, \sup\nolimits_{n} \IE W_n(x,\b)^2 < \infty.
\end{equation}
The following easy proposition allows for the definition of a threshold for the $L^2$ region, similar to $\b_c$,
\begin{proposition} \label{prop:L2}
  There is a parameter $\b_2\in [0,\infty]$ such that \eqref{eq:L2conditionIntro} holds for $\b\in [0,\b_2)$ and $\sup\nolimits_{n} \IE W_n(x,\b)^2=\infty$ for all $x\in V$ when $\b>\b_2$.
\end{proposition}
\begin{remark}
  \label{rem-birkner}
  It is immediate that $\b_2\leq \b_c$. For the lattice/SRW model, it is further known that $0< \b_2 < \b_c$ for $d\geq 3$ \cite{birkner2011collision,BirknerSun,BergerToninelli,BirknerSun11}  (in particular, see
\cite[Section 1.4]{BirknerSun} for d = 3,4). In our general setting,
we will construct graphs for which SRW satisfies
the a priori surprising property that $\b_c>0$, but $\b_2=0$,
see Theorem \ref{th:ZdwithPipes}.
\end{remark}

Compared to the full weak disorder region, the $L^2$-region has the advantage of allowing second-moment computations which, for example, lead to the first proofs of diffusivity of the path for the lattice in the $L^2$-region (and $d\geq 3$), see  Remark 3.3 in \cite{CStFlour} for a summary on the matter.
In our general context, similar considerations bring us to the following
result, whose proof is given in  Section \ref{sec:L2region}.
Recall that $|S_n|=d(S_n,S_0)$. 
\begin{theorem} \label{th:diffL2} Assume 
 \eqref{eq:L2conditionIntro}, and
  that there exist a random variable $Z$ and a
  deterministic sequence $a_n\to_{n\to\infty}\infty $ satisfying
  $a_n/a_{n-\ell} \to 1$  for all $\ell>0$, such that for
  all $x\in V$,
  \[a_n^{-1}|S_{n}| \cvlaw Z,\quad \mbox{\rm as $n\to\infty$, under $\DP_x$.}\]
 Then, for all bounded and continuous function $F$, as $n\to\infty$,
\begin{equation} \label{eq:diffL2_proba}
\DE^{n,\omega}_x\left[F\left(a_n^{-1}|S_{n}|\right)\right] \cvIP  \DE[F(Z)].
\end{equation}
\end{theorem}
\begin{remark}Recently,
the rate of convergence in $W_n \to W_\infty$ and the nature of the fluctuations for the lattice/SRW model
have been obtained in the full region $[0,\b_2)$ in \cite{CoNa20,CL16}. It is believed that the speed and nature should be different in the region $[\b_2,\b_c)$. See also \cite{LyZy20,CoNaNa20,GuRyZe18,DuGuRyZe20,MaUn18} where similar questions appear in the study of the regularized SHE and KPZ equation in dimension $d\geq 3$. We do not touch upon these questions in this paper.
\end{remark}

\section{General results}
\label{sec-2}
We state in this section our general results for the polymer model.  In Subsection \ref{subsec-crit} we give conditions for $\b_2=0$ or $\b_2>0$ in terms of quantitative
transience/recurrence of $(G,S)$, and in particular in terms of heat kernel estimates and volume growth. We also show in Theorem \ref{th:isoper} that
 recurrent walks admitting appropriate heat kernel bounds satisfy $\b_c=0$. Subsection \ref{sec:UIintro} is devoted to the weak disordered regime. We give sufficient
 conditions for  the uniform integrability of
 $W_n(x,\b)$ in terms of graph notions such as the Liouville property
 and existence of good graph isomorphisms. (We emphasize that weak disorder does not imply uniform integrability, see Proposition \ref{prop:counterExUI}.)
Along the way, we refer to counter examples to natural conjectures; these counter examples are constructed later in the paper.

Throughout, we assume without stating it explicitly
that Assumption \ref{ass:1st} holds.
Recall that a random walk on $G$ with transition probability $P(x,y)$ is called
\textit{reversible} (with reversing measure $\pi$)
if $\pi$ is a positive measure on $V$ so that
for any $x,y\in V$,
$\pi(x)P(x,y)=\pi(y) P(y,x)$.
\subsection{Critical parameters}
\label{subsec-crit}
We begin with a sufficient condition for $\beta_2=0$, in the reversible setup,
for recurrent walks.
This condition covers the (known) case of SRW on $\mathbb Z^d$ for $d=1,2$, and applies to reversible
walks.
\begin{theorem} \label{th:noL2rec}
Suppose that $(S_k)$ is recurrent and reversible with a reversing
measure $\pi$ satisfying $\inf_{x\in V} \pi(x)>0$.
Then $\b_2 = 0$.
\end{theorem}
\begin{remark}
If $(S_k)$ is a SRW then
the condition on $\pi$
in  Theorem \ref{th:noL2rec} is always satisfied, since in that case $\pi(x)=d_x\geq 1$
is a reversing measure.
\end{remark}

\begin{remark}
The condition that $(S_k)$ is recurrent on $G$ is not sufficient for the conclusion of Theorem \ref{th:noL2rec} to hold, see 
 Section \ref{sec:L2regionRec} for a counter-example.
 \end{remark}

A sufficient condition for $\b_2=0$
involves the intersection of pair of paths.
\begin{theorem} \label{th:ConditionForNoL2}
Let $(S),(S')$ denote independent copies
of $(S)$ and assume that
\begin{equation} \label{eq:1dconditionNoL2}
\sup_{x\in V} \DE^{\otimes 2}_{x}\sum_{k\geq 0} \mathbf{1}_{S_k=S_k'=x} = \infty.
\end{equation}
Then, $\b_2=0$.
\end{theorem}
\noindent Note that condition \eqref{eq:1dconditionNoL2} can be written as
\begin{equation}
\label{eq-bloc58c}
\sup_{x\in V} \sum_{k\geq 0} P_x(S_k=x)^2  = \infty
\end{equation}
A refinement of Theorem \ref{th:ConditionForNoL2} appears in
Theorem \ref{th:ConditionForNoL2bis}.
\begin{remark}
  In Proposition \ref{cor:NoL2forLongPipes} below, we
  describe a family of graphs (including transient ones) with $\b_2=0$.
  This covers the case of the supercritical percolation cluster on $\mathbb Z^d$ with
  $d\geq 2$.
\end{remark}

In the reverse direction,
we require a quantitative criterion. Introduce the
\textit{Green function} for $(S_k)$:
\begin{equation} \label{eq:defGreenFunction}
G(x,y) =  \sum_{k= 0}^\infty \DP_x(S_k=y), \quad x,y\in V.
\end{equation}
\begin{theorem} \label{th:transImpL2}
Suppose that $(S_k)$ is transient and reversible with reversing measure $\pi$
satisfying $\sup_{x\in V} \pi(x)<\infty$.
If
\begin{equation} \label{eq:conditionGreenFunction}
\sup_{x\in V} \frac{G(x,x)}{\pi(x)} < \infty,
\end{equation}
then $\b_2>0$.
\end{theorem}
\begin{remark}
\label{rem-wrong}
When $(S_k)$ is a SRW then $\pi(x)=d_x$ is a
reversing measure and the boundedness
condition in Theorem
\ref{th:transImpL2} means that the degree is uniformly bounded in $V$.
Under this condition, \eqref{eq:conditionGreenFunction} is satisfied if
and only if the Green function is bounded from above.
\end{remark}
\begin{remark} \label{rem:L2existTheorem}
The boundedness condition on $\pi$
in Theorem
\ref{th:transImpL2}
is not necessary: indeed, the biased random walk on a canopy tree is an
example of a reversible transient graph that does not satisfy this property,
while $\b_2>0$ holds for the associated polymer, see Theorem \ref{th:L2regionCanopy}.
\end{remark}

\begin{remark} Condition \eqref{eq:conditionGreenFunction} is satisfied as soon as $(S_k)$ admits a uniform Gaussian heat kernel upper bound as in \eqref{eq:gaussianBoundIsoper} below, for some $d>2$.
\end{remark}

\begin{remark}
There are polymers associated with transient and reversible SRW
that do not possess an $L^2$-region. Indeed,
in Theorem \ref{th:noL2perco} below we show that the SRW
on the supercritical percolation cluster on $\mathbb Z^d$, $d\geq 3$
(which is transient and reversible with reversing measure bounded from above),
has $\beta_2=0$ (in contrast with the full lattice SRW).
Further,
there are reversible and transient walks such that
$\b_c=0$. For an example of the latter,
take  $G=\mathbb Z_+$, with $P(i,i+1)=e^{e^{i+1}}/(e^{e^{i+1}}+e^{e^{i}})$.
(This corresponds to a conductance model with conductances  $C_{i,i+1}=e^{e^i}$.)
It is not hard to verify that 
the resulting random walk is irreducible, transient and reversible, while  a repetition of the proof of Proposition
\ref{prop:counterExUI} shows that $\bar \b_c=\b_c=0$.

There are also polymers associated with a transient
SRW satisfying $0=\b_2<\b_c$.
We exhibit two examples of this phenomenon for SRW on appropriate graphs,
namely
a class of
transient
Galton-Watson trees,
 see Section \ref{sec:transGW}, and a copy of
$\mathbb Z^d$ for $d\geq 4$ with arbitrary long pipes attached on a line,
see Theorem \ref{th:ZdwithPipes}.
\end{remark}

Returning to the general (not necessary reversible) setup,
we begin with the positive recurrent case:
\begin{theorem} \label{th:strongDesForStrongRec}
If $(S_k)$ is positive recurrent, then $\b_c=0$.
\end{theorem}
\begin{remark}
If the positive recurrent $(S_k)$
admits return times that have exponential moments, then
$\b_c=\bar \b_c=0$, see
\cite{CaGuHuMe04}.
However, an extra condition beyond positive recurrence
cannot be omitted in general. Indeed,
the positive recurrent $\lambda$-biased walk on a Galton-Watson tree with
$m<\lambda$ provides an example where $\bar \b_c$ can be positive
depending on the characteristics of the offspring distribution,
see Theorem \ref{th:posRecGW}. In particular,
this gives an  example of a polymer where $\b_c < \bar \b_c$. We note that the question whether $\b_c=\bar \b_c$ or not when $d\geq 3$ is still open in the case of the lattice/SRW $\mathbb Z^d$.
\end{remark}

We now introduce a class of walks, for which the existence  of a weak disorder region is determined by the value of the spectral dimension of the walk.
We say that a random walk $S$ satisfies a \emph{sub-Gaussian
  heat kernel upper bound} with parameters $d_f>0,d_w>1$
  if there exist a positive
  measure $\mu$ on $V$, a vertex $x\in V$ and constants $C_x,c>0$, such that
 for all $n>0$,
  \begin{equation} \label{eq:gaussianBound}
    \forall y\in V, \quad  p_n(x,y)\leq C_x n^{-d/2}  e^{-(\frac{d(x,y)^{d_w}}{c n})^{\frac{1}{d_w-1}}} \mu(y),
\end{equation}
where $d=2d_f/d_w$. 
  We say that the sub-Gaussian heat kernel upper bound is \emph{uniform} if
  in addition, $\inf_{x\in V} \mu(x)>0$ and for some $C>0$,
  \begin{equation} \label{eq:gaussianBoundIsoper}
    \forall x,y\in V, \quad   p_n(x,y)\leq C n^{-d/2}  e^{-(\frac{d(x,y)^{d_w}}{c n})^{\frac{1}{d_w-1}}} \mu(y).
\end{equation}
The estimates are called Gaussian if $d_w=2$.
The notation $d_w,d_f$ (for the walk and fractal dimensions)
is borrowed from the theory of random walks on fractals,
see e.g. \cite{KumagaiTakashi2014RWoD} for an extensive introduction. The exponent $d$, often written $d_s$ in the literature,  is referred to as the spectral dimension.

\begin{remark}
The bound  \eqref{eq:gaussianBoundIsoper} 
holds with
$d_w=2$  whenever $\mu$ is a reversing measure for $S$, such that $\inf_{x\in V} \mu(x) > 0$ and such that
$S$ satisfies the $d$-dimensional isoperimetric inequality, see
\cite[pg.\ 40 \& Section 14]{W00}.
\end{remark}

\begin{theorem} \label{th:isoper}
Assume the existence of a measure $\mu$ satisfying $\inf_{x\in V} \mu(x)>0$ and,
 with $\mathcal S(x,r) = \{y\in V, d(x,y)=r\}$,
\begin{equation}
  \label{eq:volume_growth}
  \mu(\mathcal S(x,r))\leq C_V r^{d_f-1}, \quad \text{for all $x\in V$, $r\geq 1$}.
\end{equation} 
(i) If
\eqref{eq:gaussianBound} holds with $d<2$ then $\b_c=0$.\\
 (ii) If \eqref{eq:gaussianBoundIsoper} holds for $d>2$ and $\sup_{x\in V} \mu(x)<\infty$, then $\b_2>0$.
\end{theorem}
\begin{remark}
If $\mu$ is reversing for $S$, then 
the condition \eqref{eq:volume_growth} with $d_f<2$ implies that $S$ is recurrent, see \cite[Lemma (3.12)]{W00}. 
\end{remark}
\begin{remark}
  The assumptions of Theorem \ref{th:isoper} (in fact, with the stronger
  \eqref{eq:gaussianBoundIsoper} replacing \eqref{eq:gaussianBound})
  hold for SRW on the Sierpi\'{n}ski gasket \cite{jones}
  and on the Sierpi\'{n}ski carpet \cite{Barlow,BB}. Note that one can find a family of Sierpinski gaskets with arbitrary large $d_f$ while $d<2$  \cite{HamblyKumagaiSierpinski}. The bound \eqref{eq:gaussianBoundIsoper} 
  holds for general classes of fractal graphs,  
  see Remark 4.5.3 in \cite{KumagaiTakashi2014RWoD}. 
\end{remark}


\begin{remark}
  \label{rem-CV}
  The conclusion of  Theorem
\ref{th:isoper}(i)  holds for  SRW on a graph $G$ satisfying the uniform
volume growth
$|B(x,r)|\leq Cr^2$ where $B(x,r)=\{y\in V,d(y,x)\leq r\}$,
with a uniform bound on the degree of vertices.
Indeed, for such graphs, the  Carne-Varopoulos bound (see
\cite{W00} or \cite{LPbook}) yields that
$p_n(x,y)\leq C_x e^{-d(x,y)^2/2n}$. Together with the argument
in \cite[Section 6.2.1]{CStFlour}, this immediately yields that $\b_c=0$.
This remark applies to more general walks (not necessarily reversible) satisfying the Carne-Varopoulos
bound (such as in \cite{mathieu}) and graphs satisfying quadratic volume growth.
\end{remark}
To obtain very strong disorder, we need uniform covering  conditions, of the following type.
\begin{assumption} \label{ass:covering} Suppose that there exist  $x_0\in V$ and $C_G>0$, such that for $n$ large enough, for all $m\in \mathbb N$, one can find a sequence of sets $A_i\subset V$ that satisfies $B(x_0,nm)\subset \cup_{i\in I} A_i$, $\mathrm{diam}(A_i) \leq n^{1/d_w}$ and
\begin{equation} \label{eq:unifDim}
\sup_{j\in I}\#\left\{i\in I:kn^{1/d_w}\leq  d(A_i,A_j) < (k+1)n^{1/d_w}\right\} \leq C_G\, e^{ \left(\frac{ k^{d_w}} {c_2}\right)^{\frac 1{d_w-1}}},
 \end{equation}
where $c_2>c$ and $c,d_w$ are as in the  
uniform
sub-Gaussian heat kernel upper bound  \eqref{eq:gaussianBoundIsoper}.
\end{assumption}
Assumption \ref{ass:covering} holds for many fractal graphs, such as the 
Sierpi\'{n}ski gasket and carpet, and their random variants.
\begin{theorem}
  \label{th:VSDVG}
Assume the hypotheses of Theorem \ref{th:isoper}, with the
uniform
  \eqref{eq:gaussianBoundIsoper} replacing \eqref{eq:gaussianBound},
  and in addition
let
Assumption \ref{ass:covering} hold. Then, $\bar \b_c=0$.
Moreover, there exists $C>0$ such that for all $x_0\in V$, $\b\in(0,1)$, 
\begin{equation} \label{eq:freeEnergyEstimateLacoin}
\limsup_{n\to\infty} \frac 1n \IE\log W_n(x_0,\b) \leq -C \beta^{\frac 4 {2-d}}.
\end{equation}
\end{theorem}
\begin{remark}
  \label{rem-Lacoin}
  The critical
  case $d=2$ is not covered by Theorems \ref{th:isoper} and
  \ref{th:VSDVG}. For $\mathbb Z^d$/SRW, the conclusion holds by
  \cite{Lacoin10}. Unlike the proof  of Theorems
  \ref{th:isoper} and \ref{th:VSDVG} for $d<2$,
  the proof for $d=2$ in \cite{Lacoin10} uses a change of measure
  that introduces correlations into the environment.
  We believe that the argument carries over to our setup, but we have not verified all details.
\end{remark}


We close this section by mentioning a result of Birkner \cite{Bir04}
whose proof
carries over without changes to our general framework.
Let
\begin{equation}
  \label{eq-lambda2}
  \Lambda_2(\b) := \Lambda(2\b)-2\Lambda(\b).
\end{equation}
\begin{theorem}[\cite{Bir04}]\label{th:Birkner} Let $x\in V$. Let $S,S'$
  be two independent copies of $S$ started at $x$, and
  let $\mathcal F_S$ denote the $\sigma$-algebra generated by $(S_k)$. If
\begin{equation} \label{eq:BirknerCondition}
 \DE_x^{\otimes 2}\left[ e^{\Lambda_2(\b) \sum_{k=1}^\infty
 \mathbf{1}_{S_k = S_k'}} \middle| \mathcal F_S\right] < \infty,\quad \mbox{\rm a.s.},
\end{equation}
then
$(W_n(x,\b))_n$ is uniformly integrable. In particular, \eqref{eq:BirknerCondition}
implies that $W_\infty(x,\b) >0$ a.s.
\end{theorem}
Theorem \ref{th:Birkner} was used in the proof that
$\b_2 < \b_c$ for the $\mathbb Z^d$/SRW polymer,
when $d\geq 3$ (see Remark \ref{rem-birkner} above).
In our context,
it will be used in Section \ref{sec:transGW}
when showing that $\b_c >0$  for a transient Galton-Watson model
which satisfies $\b_2=0$.

\subsection{Uniform integrability of the partition function} \label{sec:UIintro}
Since $\IE[W_n(x,\b)]=1$, it is immediate that whenever $(W_n(x,\b))_n$ is uniformly integrable for some $x\in V$, then weak
disorder holds. In what follows, we study the converse
implication and provide
some conditions on $(S_k)_{k\geq 0}$ under which $(W_n(x,\b))_n$ is uniformly integrable in the entire weak disorder region.
\begin{remark}
The converse is not always true: Proposition \ref{prop:counterExUI} below provides an example for which $\b_c>0$ but $(W_n(x,\b))_n$ is not uniformly integrable in the whole weak disorder region. 
\end{remark}

 We begin with an observation. We say that $h:V\to\mathbb R$ is an \emph{harmonic function} on $G$ whenever
\begin{equation} \label{eq:harmonic}
\Delta h(x) = \sum_{y\sim x} P(x,y) (h(y)-h(x)).
\end{equation}
\begin{proposition} 
\label{prop:UIcond}
\begin{enumerate}
\item  $h(x)= \IE[W_\infty(x,\b)]$ defines a bounded harmonic function on $G$. \item The following properties are equivalent:
\begin{enumerate}[label=(\roman*)]
\item \label{UIi} $(W_n(x,\b))_{n}$ is uniformly integrable {for some $x\in V$,}
\item \label{UIii} $\IE[W_\infty(x,\b)] = 1$ {for some $x\in V$,}
\item \label{UIiii} $\inf_{x\in V} \IE[W_\infty(x,\b)] > 0$.
\end{enumerate}
\item  If one of the above properties is satisfied then for all $x\in V$, $(W_n(x,\b))_{n}$ is uniformly integrable and $\IE[W_\infty(x)]=1$ .
\end{enumerate}
\end{proposition}

\begin{corollary} \label{cor:Liouville}
Assume that $(G,P)$ satisfies the Liouville property, i.e. that
all bounded harmonic functions are constant. Then, $(W_n(x,\b))_{n}$
is uniformly integrable for all $x\in V$ in the whole weak disorder region.
\end{corollary}
In what follows, given a graph $G$ and a vertex $v\in V$, we call the
pair $(v,G)$ a \textit{rooted graph}. We say that two rooted graphs
  $(v,G)$ and $(v',G')$
  are isomorphic if there exists a graph isomorphism $\pi$ so that
$v'=\pi(v)$ and $G'=\pi(G)$.
\begin{corollary} \label{cor:UIbyTranslation}
Suppose there is a finite set of vertices $V_0\subset V$ such that for all $x\in V$, the rooted graph
  $(x,G)$ is isomorphic to one of the rooted graphs $\{(v,G)\}_{v\in V_0}$. Then, if weak disorder holds, $(W_n(x))_{n}$ is uniformly integrable for all $x\in V$.
\end{corollary}

\section{Specific graphs}
\label{sec-spec}
We introduce in this short section
three models, which will be used to illustrate various
phenomena. These are respectively SRW on the lattice infinite
bond percolation cluster, the $\lambda$-biased random walk on a Galton--Watson tree, and the canopy tree.
\subsection{Super-critical percolation cluster on $\mathbb Z^d$}
\label{subsec-perco}
To each edge $(x,y)$
of the lattice $\mathbb Z^d$ we associate a
Bernoulli random variable $a_{xy}$ such that the edge is open (i.e.\ $a_{xy}=1$) with probability $p$. It is well known, see e.g. \cite{grimmett},
that
for $d\geq 2$ there exists a
critical parameter $p_c=p_c(\mathbb Z^d) \in (0,1)$, such that for
the super-critical regime $p>p_c$,
there exists almost-surely
a unique infinite connected cluster denoted by $\mathcal C_\infty$.

The SRW on the super-critical infinite cluster shares properties similar
to the SRW on $\mathbb Z^d$; indeed,
the walk on $\mathcal C_\infty$ is transient when $d\geq 3$ and recurrent when $d=2$,
almost surely \cite{GKZ93}. The SRW further
satisfies almost surely  a central limit theorem and  a local limit theorem
\cite{barlowperc}. We will however see that the polymer measure on the
percolation model is quite different, and in particular, see Theorem
\ref{th:noL2perco} below, does not possess an
$L^2$ regime.

\subsection{$\lambda$-biased
random walk on Galton-Watson trees} \label{sec:GWtree}
Let $\mathcal T$ be a rooted (at a vertex $o$) Galton-Watson tree (conditioned on survival)
with offspring distribution $\{p_k\}$,
having mean $m=\sum k p_k>1$. The parent of $x\in V$ is the neighbor of $x$
on the geodesic connecting $x$ to the root. All other neighbors of $x$ are
called descendents. Given a real $\lambda\geq 0$, we let $(\DP,(S_k))$ denote
the $\lambda$-biased random walk on $\mathcal T$, with transition probability
$P(x,y)=\lambda/(\lambda+d_x)$ if $y$ is the parent of $x$ and
$P(x,y)=1/(d_x+\lambda)$ otherwise.
Lyons  \cite{lyons1990random} proved that the walk is transient if $m>\lambda$,
null recurrent if $m=\lambda$ and positive recurrent if $m<\lambda$. Let
$|S_k|$
denote the distance of $S_k$ from the root. Law of
large numbers for $|S_n|/n$, based on appropriate regeneration structures,
were derived in \cite{LyonsPemantlePeresBiased} (for the transient case).
These were completed by large deviation principles in
\cite{dembo2002large}, and by central limit theorems (for the transient $m>\lambda>0$ and null-recurrent $m=\lambda>0$ cases) in  \cite{peres2008central}.
Note that the model of Bernoulli
percolation on a Galton-Watson tree is a particular case of this model. Note also that the case $\lambda=0$ corresponds to the model of Branching random walk, and
$W_n(o)$ is then the  Biggins martingale. Much is known about $W_n(o)$ and its limits, see \cite{Shi-StFl}.

Our results for polymers with the biased random walk on Galton-Watson trees are presented in Sections \ref{sec:posRecGW} and \ref{sec:transGW}.

\subsection{The canopy tree}
 \label{sec:canopyTree}
The canopy tree
$\mathtt T$
is the infinite volume limit of a finite $d+1$-regular tree seen from its bottom boundary \cite{AiWa06Canopy}. It is constructed
as follows. At the ground level $\ell=0$, put a countable number of vertices and attach to every successive pack of $d$ vertices one parent at level $\ell=1$. Do the same recursively at the higher $\ell$ levels (see Figure \ref{fig:canopy} for a pictorial representation).

The  $\lambda$-biased random walk
(with bias toward the parents) is constructed similarly to
section \ref{sec:GWtree}. Namely, when at vertex $v$ at level $l>0$,
the jump probability toward the parent is $\lambda/(\lambda+d)$ while the probability to jump to any other neighbor is $1/(\lambda+d)$. When at vertex $v$ at level $0$, the jump probability toward the parent of $v$ is $1$.
It follows from the description that the $\lambda$-biased
walk can be represented  in terms of a conductance model (see \cite{LPbook})
with the conductance on edges
between levels $\ell$ and $\ell+1$ equal to $\lambda^\ell$.
From this representation it follows at once that the
$\lambda$-biased walk is recurrent if $\lambda < 1$, null recurrent if $\lambda = 1$ and transient when $\lambda > 1$. 

 We show in Section \ref{sec:canopyProof} that $\b_2>0$ for the polymer on $(\mathtt T,(S_k))$ whenever $d>\lambda > 1$.
\begin{figure}
  \includegraphics[width=\linewidth]{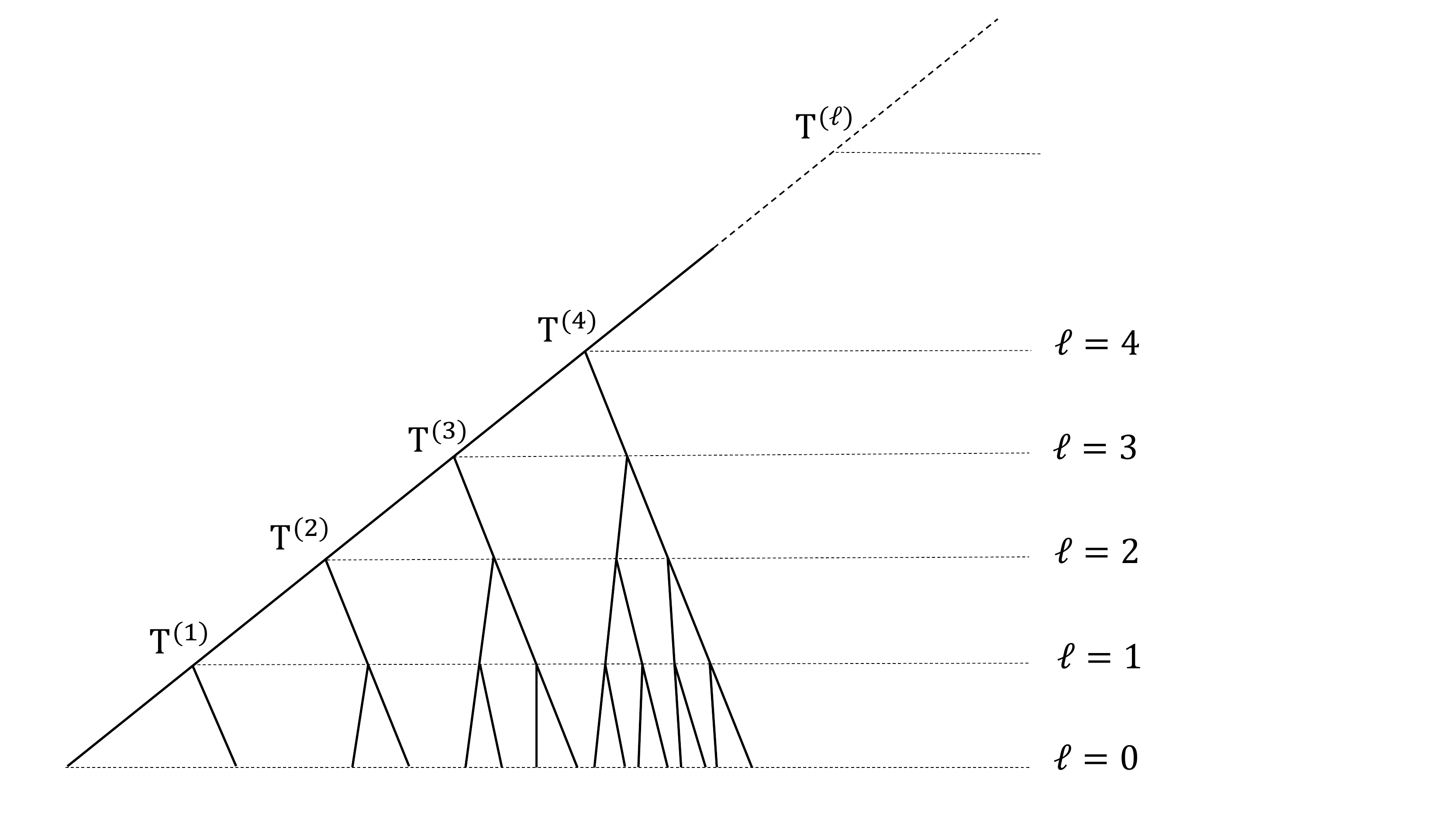}
  \caption{The canopy tree of parameter $d=2$. The
  subtrees $\mathtt T^{(\ell)}$ are used in the
  proof of Theorem \ref{th:L2regionCanopy}}
  \label{fig:canopy}
\end{figure}

\section{Weak, strong and very strong disorder}
\label{sec-3}
We provide in this section proofs for our main results. 
\subsection{Proof of Proposition \ref{prop:StrongWeakDisorderIntro}}
\begin{proof}[Proof of Proposition \ref{prop:StrongWeakDisorderIntro}]
The proof mostly follows  classical techniques of directed polymers, see e.g.\ \cite{CStFlour}.
Let us first introduce the short-notation $e_n := e^{\b \sum_{i=1}^n \omega(i,S_i)-n\Lambda(\b)}$ such that $W_n(x) = \DE^x_n[e_n]$ that we will use repeatedly.
Let $n,m\in \mathbb N$ and denote by $\eta_n$ the shift in time of $n$ steps in the environment $\omega(i,y)$, and observe that by Markov's property,
\begin{align}
 W_{n+m}(x) &= \DE_x \left[ e_n \, e^{\b \sum_{i=n+1}^{n+m} \omega(i,S_i)- m\Lambda(\b)}\right]\nonumber\\
 & = \DE_x\left[e_n\, W_m(S_n)\circ \eta_n \right] =  \sum_{y\in R_n(x)} \DE_x\left[e_n \mathbf{1}_{S_n=y}\right]  W_m(y)\circ \eta_n,\label{eq:finiteMarkovWn}
\end{align}
where $R_n(x)= \{y\in V, p_n(x,y)>0\}$. As $R_n(x)$ is finite by assumption, we obtain after taking the limit $m\to\infty$ that
\begin{equation} \label{eq:WinftySelfEq}
 W_\infty (x) =  \sum_{y\in R_n(x)} \DE_x\left[e_n \mathbf{1}_{S_n=y}\right] W_\infty(y)\circ \eta_n.
\end{equation}
Since all the $\DE_x\left[e_n \mathbf{1}_{S_n=y}\right],y\in R_n(x)$ are almost surely positive,
\begin{equation} \label{eq:01lawproof}
\{W_\infty(x) = 0\} = \{\forall y\in R_n(x), \ W_\infty(y)\circ \eta_n = 0  \},
\end{equation}
hence $\{W_\infty(x) = 0\}$ is measurable with respect to $\mathcal G_n^+ := \sigma(\omega(i,x), i\geq n, x\in V)$ for all $n$ hence it is a tail event, so by Kolmogorov 0-1 law either $W_\infty(x) = 0$ a.s.\ or $W_\infty(x) > 0$ a.s. Now the first statement of Proposition \ref{prop:StrongWeakDisorderIntro} follows from \eqref{eq:01lawproof} and the irreducibility assumption.

We turn to the second part of the proposition, namely the existence of the critical parameter $\b_c$.
Let $\theta\in(0,1)$ and $x\in V$.
We begin by observing that $\IE[W_\infty(x,\b)^\theta]$ is non-increasing in $\b$, for which it is enough to prove that $\b \to \IE[W_n(x,\b)^\theta]$ is non-increasing for all $n$ 
(note that $(W_n(x,\b)^\theta)_n$ is uniformly integrable since $\IE[W_n(x,\b)]=1$). 

Let $H_n(S) = \sum_{i=1}^n \omega(i,S_i)$ and  $e_n(S) = e^{\b H_n(S) - n\Lambda(\b)}$.
 For all $n\geq 0$, we have by Fubini:
\begin{align*}
\frac{\dd}{\dd \b} \IE\left[W_n(x,\b)^\theta\right] & = \theta \DE_x\IE[e_n(S) W_n(x,\b)^{\theta-1} (H_n(S) - n\Lambda(\b)')  ]\\
& = \theta \DE_x \IE^{S,n}\left[W_n(x,\b)^{\theta-1} (H_n(S) - n\Lambda(\b)')\right],
\end{align*}
where $\dd\IP^{S,n} = e_n(S) \dd \IP$. The field $\omega(i,x)$ is still independent under $\IP^{S,n}$, so by the FKG inequality for independent \ random variables \cite{Harris}, we see that since $H_n(S)$ is non-decreasing with respect to the environment while $W_n(x,\beta)^{\theta - 1}$ is non-increasing, we have
\begin{align*}
\frac{\dd}{\dd \b} \IE\left[W_n(x,\b)^\theta\right] & \leq \theta \DE_x\left[ \IE^{S,n}\left[W_n(x,\b)^{\theta-1}\right] \IE^{S,n}[H_n(S) - n\Lambda(\b)']
\right]\\
& = \theta \DE_x\left[ \IE^{S,n}\left[W_n(x,\b)^{\theta-1}\right] \frac{\dd}{\dd \b} \IE [e_n(S)]
\right] = 0,
\end{align*}
where the last equality holds since $\IE [e_n(S)] = 1$.

Now, observe that
\[\b_c(x) := \inf \left\{\b \geq 0 : \IE\left[W_\infty(x,\b)^\theta\right] = 0\right\},\]
does not depend on $x$. Indeed, if $\b_c(x)= \infty$ for some $x$, then by \eqref{eq:01lawproof} we have that  $W_\infty(y,\b)>0$ a.s.\ for all $y$ and all $\b\geq 0$, hence $\b_c(y)=\infty$ for all $y$. Similarly, if $\b_c(x)<\infty$ for some $x$, then all $\b_c(y)$ are finite. In this case, let $x,y\in V$ and $\b>\b_c(x)$, so that $W_\infty(y,\b) = 0$ for all $y \in V$ and thus $\b_c(y)\leq \b_c(x)$; by exchanging the role of $x$ and $y$, $\b_c(y) = \b_c(x)$.

We  conclude the proof  by  checking, similarly to what we just did, that $\b_c:=\b_c(x)$ separates strong disorder from weak disorder.
\end{proof}

\subsection{Uniform integrability: proofs for Section \ref{sec:UIintro}}
 \begin{proof}[Proof of Proposition \ref{prop:UIcond}]
1. Let $h(x) = \IE[W_\infty(x)]$.
By Fatou's lemma, the expectation $\IE[W_\infty(x)]$ is uniformly bounded by $1$.
To see that $h$ is harmonic, take expectation with respect to the environment in \eqref{eq:WinftySelfEq} and let $n=1$.

2.  We begin with the implication \ref{UIi} $\Rightarrow$ \ref{UIii}. If $(W_n(x))_n$ is uniformly integrable, then $W_n(x)$ converges in $L^1$ and so $\IE[W_\infty(x)]=1$ since $\IE[W_n(x)]=1$.

To see that \ref{UIii} implies \ref{UIiii}, suppose that $h(x)=\IE[W_\infty(x)]=1$. Since $h$ is bounded by $1$, this implies that $h$ is a local maximum. By harmonicity, $h$ must be constant equal to 1
which is \ref{UIiii}.

We now turn to \ref{UIiii} $\Rightarrow$ \ref{UIi}. Suppose that $\inf_{y\in V} \IE[W_\infty(y)] > 0$ and let any $x \in V$. By identity \eqref{eq:WinftySelfEq}: (recall the definition of $\mathcal G_k$ in the introduction)
\begin{align*}
\IE\left[ W_\infty(x) \middle| \mathcal G _k\right] & = \sum_{y\in V} \DE_{x}\left[ e_k \mathbf{1}_{S_k = y}\right] \IE[W_\infty(y)]\geq \inf_{y\in V} \IE[W_\infty(y)] W_k(x),
\end{align*}
and since $\IE\left[ W_\infty(x) \middle| \mathcal G _k\right]$ is UI, $(W_n(x))_n$ is also UI.

Point 3 of the proposition follows from the last arguments.
\end{proof}
\begin{proof}[Proof of Corollary \ref{cor:Liouville}]
By  point 1 of  Proposition \ref{prop:UIcond}, $h$ is constant when $(G,P)$ satisfies the Liouville property.
Hence $\inf_x  h(x)$ is non-zero since weak disorder holds and so uniform integrability holds  by point 3 of Proposition \ref{prop:UIcond}.
\end{proof}

\begin{proof}[Proof of Corollary \ref{cor:UIbyTranslation}]
Under the assumption of the corollary the function $h(x)=\IE[W_\infty(x)]$ only takes a finite number of values. Since weak disorder holds, all of them are positive and the statement follows from Proposition \ref{prop:UIcond}.
\end{proof}

We next construct an example of a pair $(G,S)$ where $\b_c>0$ but $W_n(x,\beta)$ is not uniformly integrable for any $\beta>0$ and some $x\in V$. Let $\mathcal{T}_d$ denote the
$d$-ary tree
 rooted at a vertex $o$. Augment $\mathcal{T}_2$  by attaching to $o$ a copy of $\mathbb Z_+$, to created a rooted tree $\mathcal{G}_{\ref{prop:counterExUI}}$,
 see figures \ref{fig:twotrees}.
  \begin{figure}
  \includegraphics[scale=0.3]{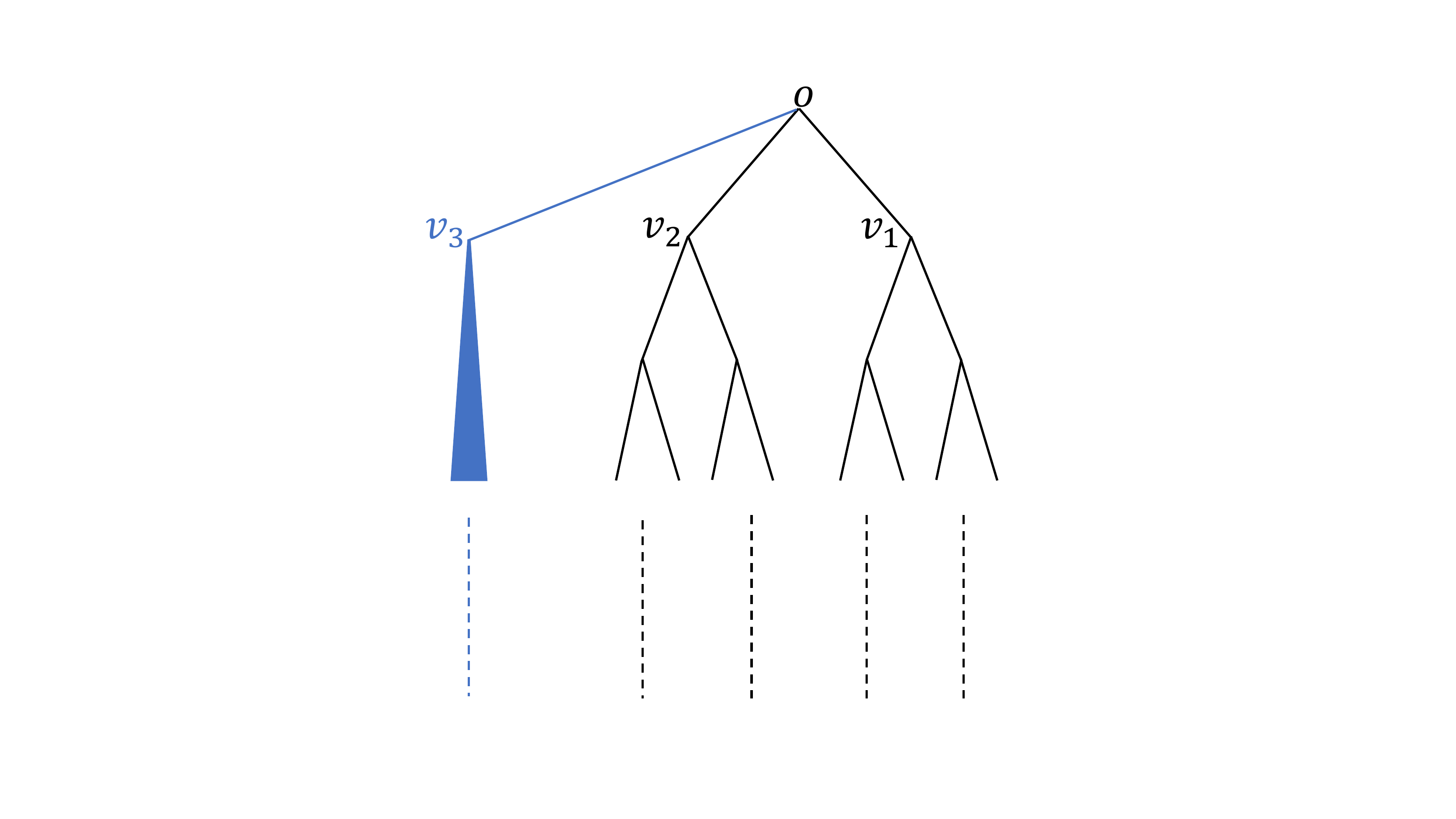}
  \caption{The graph $\mathcal{G}_{\ref{prop:counterExUI}}$, for which $W_n(x,\b)$ is not uniformly integrable for any $\b>0$, but $\b_c>0$. The increasing width of the leftmost  ray represents increasing conductance. }
  \label{fig:twotrees}
\end{figure}
 To define the random walk on $\mathcal{G}_{\ref{prop:counterExUI}}$, assign to each edge $e$ of $\mathcal{T}_2$ the conductace $C_e=1$, while to the $i$-th edge of $\mathbb Z_+$ (measured from the root), assign the conductance $C_i=e^{e^i}$. For $(x,y)\in E$,
write $C_{x,y}$ for the conductance of the edge $(x,y)$, and set $P(x,y)=C_{x,y}/\sum_{z\sim x} C_{x,z}$.
\begin{proposition} \label{prop:counterExUI} The polymer on $\mathcal{G}_{\ref{prop:counterExUI}}$ with i.i.d. bounded
 $\omega(x,i)$
has  $\b_c>0$ while for any $\b>0$,  $W_n(o,\b)$ is not uniformly integrable.
\end{proposition}
\begin{proof}
Let $v_1,v_2$ be the descendents of $o$ that belong to $\mathcal{T}_2$ and let $v_3$ denote the descendent of $o$ that belongs to
$\mathbb Z_+$.
 Let
\[
\mathcal C = \{\text{the walk never visits $v_1\cup v_2$}\}.
\] Clearly, by the transience of the walk on $\mathbb Z_+$ with our conductances,
$\DP(\mathcal C)\in(0,1)$.  Writing $W_n=W_n(o,\beta)$, decompose
$W_n = W_n^{t} + W_n^{l},$
 with $W_n^l = \DE[e_n \mathbf 1_{\mathcal C}] = \DP(\mathcal C) \DE[e_n \mathbf |{\mathcal C}]$.
 Set
 \begin{equation} \label{eq:defEventAn}
 \mathcal{A}_n =\{ |X_{k+1}|=|X_k|+1, \forall k>\sqrt{n}\}.
\end{equation}
 We have that
 $\DP(\mathcal A_n^c|\mathcal C) < e^{-c n}$ for all $n$ large  (since the conditional walk has a conductance representation with increasingly strong drift away from the root).
 We now claim that for all $\beta>0$,
  \begin{equation} \label{eq:vanishWnr}
 W_n^l\to_{n\to\infty} 0,\quad \IP\text{-a.s.}
 \end{equation}
 Indeed,
 by a union bound, using that the $\omega(x,i)$ are bounded,
$$ \sup_{a,b: |a|,|b|\leq n}\sup_{j\leq \sqrt{n}} \sup_{\ell\leq \sqrt{n}}
\sum_{i=1}^{n-\ell}  \omega(i+\ell, (j+i,a,b)) =o(n),$$
and thus we have that
\begin{equation} \label{eq:vanishQuestion1a}
  W_n^l \leq CP(\mathcal{A}_n^c|\mathcal C)e^{\beta c^* n}X_n+ e^{-\Lambda(\beta) (n-\sqrt{n}) +o(n)},
\end{equation}
where $EX_n\leq 1$. Hence, $W_n^l\to_{n\to\infty} 0$, in probability, and hence $\IP$-a.s. since $W_n^l$ is a martingale.
On the other hand,
$W_n^t = \DE[e_n \mathbf 1_{\mathcal C^c}] \geq W_n^{t,*}$ where  $W_n^{t,*} = \DE[e_n \mathbf{1}_{\mathcal D}]$ with
 \[\mathcal D = \{S_1\in \{v_2,v_3\}, S_k\neq o\;  \mbox{\rm for all $k\geq 1$}\}.\]
Note that $W_n^{t,*}$ is a martingale and that
\[ \IE(W_n^{t,*})^2  \leq \IE(W_n^{\mathcal T_2}(o,\beta))^2,\]
where $W_n^{\mathcal T_2}(o,\beta)$ is the normalized partition function for the SRW on $\mathcal T_2$. We now claim that  for $\beta>0$ small enough,
\begin{equation}
  \label{eq-sep26}
  \limsup_{n\to\infty} \IE(W_n^{\mathcal T_2}(o,\b))^2<\infty,
\end{equation}
 see
e.g. \cite[(3.3b)]{DerridaSpohn88} for a similar computation.
 Indeed, letting $S,S'$ denote two independent copies of
 $S$, the SRW on ${\mathcal T}_2$ started at the root,
 and setting $N_\infty =\sum_{i=1}^\infty {\bf 1}_{S_i=S_i'}$,
 it is easy to check that there exists a constant
 $c$ so that $\DE(e^{cN_\infty})<\infty$.
As a consequence,
with $\Lambda_2(\b)$ as in \eqref{eq-lambda2}, we have
that $\Lambda_2(\b)\to_{\beta \to 0} 0$, and
 $\IE(W_n^{\mathcal T_2}(o,\b))^2\leq \DE(e^{\Lambda_2(\beta) N_\infty})<\infty$
 for
 $\beta>0$ small enough, yielding
 \eqref{eq-sep26}. Thus,
we conclude that for small $\b$,
$W_n^{t,*}$ converges a.s.\ and in $L^2$ to a strictly positive limit $W_\infty^{t,*}$ with $\IE W_\infty^{t,*}=
 \DP[\mathcal D] >0$.
Thus on the one hand, $W_\infty = \lim_n W_n$ is, for small $\b$, positive with positive probability, which implies that it is in fact positive a.s., and thus $\b_c>0$. On the other hand,
we have by \eqref{eq:vanishWnr} that for all $\b\in (0,\b_0)$, $W_\infty = W_\infty^t$, where by Fatou's lemma $\IE[W_\infty^t]\leq\DP(\mathcal C)<1$. Therefore, $(W_n)_{n\geq 0}$ cannot be uniformly integrable in $(0,\b_0)$, and hence
for any $\b>0$, since $\IE[W_n]=1$.
 \end{proof}

\subsection{ Recurrence and heat kernel bounds - Proof of Theorems \ref{th:strongDesForStrongRec} and \ref{th:isoper}}
\begin{proof}[Proof of Theorem \ref{th:strongDesForStrongRec}]
  The proof is based on the change of measure technique introduced in
\cite{Lacoin10}. By Assumption \ref{ass:1st} and our convention that
for any $(i,x)$,
$\omega(i,x)$ has zero mean and variance $1$, we obtain that for $\delta>0$,
\begin{equation}
  \label{eq-chom}
  e^{\Lambda(-\delta)}=\IE(e^{-\delta \omega(i,x)})=e^{c_\delta \delta^2/2},\quad
  c_\delta=1+o_\delta(1).
\end{equation}
We fix $x\in V$ and let $W_n=W_n(x)$.
Thanks to positive recurrence and
the ergodic theorem, we can find an $\varepsilon\in (0,1)$
such that under $\DP_x$, the process $(S_k)$ spends at most a fraction
$1-\varepsilon$ of time away from $x$, i.e.
\begin{equation} \label{eq:definitionAnEps}
  \DP_x\left( \sum_{k=1}^n \mathbf{1}_{S_k\neq x} \geq  (1-\varepsilon)n\right) \to_{n\to\infty} 0.
\end{equation}
Now, fix $\delta_n = n^{-1/2}$  and define the measure
$\IPtilde$ by
\begin{equation}
  \label{eq-chom1}
  \frac{d\IPtilde}{d\IP}=\prod_{(i,y)\in \mathcal C} e^{-\delta_n \omega(i,y)-\Lambda(-\delta_n)},
\end{equation}
where  $\mathcal C = [1,n]\times \{x\}$.
 Under $\IPtilde$, the  variables
 $(\omega(i,y))$ are independent, of  mean $-\delta_n  \mathbf{1}_{(i,y)\in \mathcal C} (1+o_n(1))$ and variance $1+o_n(1)$.
 Further, for $(i,y)\in \mathcal C$,
 \begin{align}
   \label{eq-chom2}
   \nonumber
   e^{-\Lambda(\beta)}\IEtilde e^{\beta \omega(i,y)}&=
   e^{\Lambda(\beta-\delta_n)-\Lambda(\beta)-\Lambda(-\delta_n)}=
   e^{-\Lambda'(\beta)\delta_n +O(\delta_n^2)}\\
   &= e^{-\Lambda'(\beta)\delta_n (1+o_n(1))}.
 \end{align}
Then, for any $\alpha\in(0,1)$, H\"older's inequality yields that
\begin{align}
  \label{eq-chom1a}
\IE\left[W_n^\alpha\right] & = \IEtilde\left[ \frac{\dd \IP}{\dd \IPtilde} \times W_n^\alpha \right]
 \leq \IEtilde\left[ \left(\frac{\dd \IP}{\dd \IPtilde}\right)^{\frac{1}{1-\alpha}}\right]^{1-\alpha} \times \IEtilde\left[ W_n \right]^\alpha.
\end{align}
The first term on the right hand side of the last display reads
\[\IE\left[ \left(\frac{\dd \IP}{\dd \IPtilde}\right)^{\frac{\alpha}{1-\alpha}}\right]^{1-\alpha} = e^{\frac{\#\mathcal C}{2}  \frac{\alpha \delta_n^2 }{1-\alpha}(1+o_n(1))}= e^{\frac{\alpha}{1-\alpha}(1+o_n(1))}<\infty.\]
On the other hand, by \eqref{eq-chom2},
\begin{align}
  \nonumber
  \label{eq-chom2a}
\IEtilde\left[ W_n \right] & = \DE_x\left[ e^{-\Lambda'(\beta) \delta_n (1+o_n(1))
\sum_{k=1}^n \mathbf{1}_{S_k=x}}\right]\\
 &\leq  e^{-\varepsilon \Lambda'(\beta) \delta_n(1+o_n(1)) n } + \DP_x\left( \sum\nolimits_{k\leq n} \mathbf{1}_{S_k\neq x} \geq  (1-\varepsilon)n\right).
\end{align}
By our choice of $\delta_n$ and \eqref{eq:definitionAnEps}, we conclude that
$\IEtilde\left[ W_n \right] \to_{n\to\infty} 0$.
%
Putting things together, we find that  $\IE [W_n^\alpha]\to_{n\to\infty} 0$
for all $\b>0$, and since $W_n^\alpha$ is uniformly integrable,
necessarily $W_\infty = 0$ and thus $\b_c=0$.
\end{proof}

\begin{proof}[Proof of Theorem \ref{th:isoper}]
  The proof of point (i) parallels that of Theorem \ref{th:strongDesForStrongRec}, and we use
  similar notation, with the change that now
$\delta_n = C_1^{-d_f/2} n^{-\frac{d_w+d_f}{2d_w}}$ where $C_1>0$ is a parameter to be determined later, and \[
\mathcal C = \left\{(i,y) : i\leq n,\  d(x,y) \leq C_1 n^{1/d_w}\right\}.
\]
Note that $n\delta_n ={ C_1^{-d_f/2}}n^{1/2-d/4}\to_{n\to\infty} \infty$ since $d<2$.
Introduce the measure
$\IPtilde$ as in \eqref{eq-chom1}. Proceed as in
\eqref{eq-chom1a} and, using that $|B(x,n)|\leq C n^{d_f}$
for some $C>0$ and all $n\geq 1$ by the hypotheses,
bound the first term in the right hand side of \eqref{eq-chom1a},
for all $n$ large,
by
\[\IE\left[ \left(\frac{\dd \IP}{\dd \IPtilde}\right)^{\frac{\alpha}{1-\alpha}}\right]^{1-\alpha} = e^{\frac{\#\mathcal C}{2}  \frac{\alpha \delta_n^2 }{1-\alpha}(1+o_n(1))}\leq
e^{\frac{\alpha C}{(1-\alpha)}}.\]
On the other hand, as in \eqref{eq-chom2a}, we have that
\begin{align}
  \label{eq-2709c}
\IEtilde\left[ W_n \right] & = \DE_x\left[ e^{-{\Lambda'(\beta) \delta_n (1+o_n(1))}  \sum_{k=1}^n \mathbf{1}_{d(S_k,x)\leq C_1n^{1/d_w}}}\right]\\
& \leq  e^{-\frac{\Lambda'(\beta)}{2} \delta_n n (1+o_n(1))} + \DP_x\left( \sum\nolimits_{k\leq n} \mathbf{1}_{d(S_k,x)> C_1n^{1/d_w}} \geq  n/2\right),
\nonumber
\end{align}
where the first term of the RHS in the last display
goes to $0$ as $n\to\infty$, while by Markov's inequality and estimate \eqref{eq:gaussianBound}, we find that  the second term is bounded from above by
\begin{align*}
  \frac{2}{ n} \sum_{k=1}^n \sum_{d(y,x)> C_1 n^{1/d_w}} p_k(x,y) & \leq \frac{C}{n}
   \sum_{k=1}^n \sum_{p> C_1 n^{1/d_w}} \mu\left(\mathcal S(x,p)\right)\,
   \frac{ e^{-(\frac{p^{d_w}}{ck})^{\frac{1}{d_w-1}}}}{k^{d/2}} \\
& \leq C F(C_1),
\end{align*}
for some positive $F$ such that $F(x)\to 0$ as $x\to\infty$, where the estimate in the second line holds for $n$ large enough 
and is obtained by Riemann approximation 
using that $\mu\left(\mathcal S(x,p)\right) \leq C_V p^{d_f-1}$ and $d=2d_f/d_w <2$.
We can now fix $C_1$ large enough to make RHS of the last display as small as we wish.

Putting things together, we find that for all $\b>0$, $\IE[W_n^\alpha] \to_{n\to\infty} 0$  and therefore $W_\infty = 0$ a.s.


We turn to the proof of  Theorem \ref{th:isoper} (ii). It is enough to check that condition \eqref{eq:Khas} is verified.  Note that $\DE_{x,x}^{\otimes 2} [ N_\infty(S,S')] = \sum_{n=0}^\infty \sum_{y\in V} p_n(x,y)^2$. By \eqref{eq:gaussianBoundIsoper} and \eqref{eq:volume_growth}, and using the uniform upper bound on $\mu$,
\begin{align*}
\sup_{x\in V} \sum_{y\in V} p_n(x,y)^2 & \leq  n^{-d}\sum_{k=0}^\infty C_V  k^{d_f-1} C e^{-2\left(\frac{k^{d_w}}{cn}\right)^{\frac{1}{d_w-1}}}=O( n^{-d}  n^{d_f/d_w}).
\end{align*}
 The fact that $d_f/d_w = d/2$ with $d>2$ yields \eqref{eq:Khas}.
 \end{proof}
\subsection{Very strong disorder: proof of Propositions \ref{prop:veryStrongDis} and \ref{prop:veryStrongDisAlwaysHolds}, and Theorem \ref{th:VSDVG}} \label{sec:veryStrongDisorder}

\begin{proof}[Proof of Proposition \ref{prop:veryStrongDis}]
We claim that if $\bar p(x)$ in \eqref{eq:veryStrongDisordermain}
satisfies $\bar p(x)<0$
for some $x\in V$, then the same holds for all $x\in V$, i.e.
\eqref{eq:veryStrongDisordermain} holds. Indeed, let $x,y\in V$ and $m\geq 0$ such that $y\in R_m(x)$ where $R_m(x)=\{z\in V, p_m(x,z) > 0\}$.
From \eqref{eq:finiteMarkovWn}, we see that for all $n\geq 0$,
\[W_{n+m}(x) \geq \DE_{x}[e_m \mathbf{1}_{S_m=y}] W_n(y)\circ \eta_{m},\]
hence
\[
\limsup_{n\to\infty} \frac{1}{n} \IE \log W_n(x) \geq \limsup_{n\to\infty} \frac{1}{n} \IE \log W_n(y),
\]
which justifies our claim. We now complete the proof of Proposition \ref{prop:veryStrongDis}.
Let $H_n(S) = \sum_{i=1}^n \omega(i,S_i)$. We have by Fubini:
\begin{align*}
\frac{\dd}{\dd \b} \IE\left[\log W_n(x,\b)\right] & =  \DE^x\IE[e_n(S) W_n(x,\b)^{-1} (H_n(S) - n\Lambda(\b)')  ]
\end{align*}
and following the same arguments as in the proof of Proposition \ref{prop:StrongWeakDisorderIntro}, we find that for all $n\geq 1$, $\b\to \IE \log W_n(x,\b)$ is non-increasing. Therefore $\bar p(x)$ of
\eqref{eq:veryStrongDisordermain} is also non-increasing with $\b$ and this concludes the proof.
Finally, property \eqref{eq:quenchedVerystrong} follows from the concentration inequality in Theorem \ref{th:ConcentrInequality} and the Borel-Cantelli lemma.
\end{proof}

We now turn to the:
\begin{proof}[Proof of Proposition \ref{prop:veryStrongDisAlwaysHolds}]
Let $d>0$ such that the degree of all vertices is bounded by $d$.
Let $I(a)=\sup_{\theta>0}\left\{\theta a -\Lambda(\theta) \right\}$ be the rate function of the $\omega(i,y)$'s. For any nearest-neighboor path of vertices $\mathbf{x}=(x_0,\dots,x_n)$, we note $H_n(\mathbf x) = \sum_{k=1}^n \omega(k,x_k)$. Let $a>0$ and consider
\begin{align}
\label{eq-corona1}
\sum_{n\geq 0} \IP \left(\sup_{\mathbf{x} = x_0,x_1,\dots,x_n} H_n(\mathbf x) > na\right)
& \leq \sum_{n\geq 0}  \sum_{\mathbf x = x_0,x_1,\dots,x_n} \IP \left( H_n(\mathbf x) > n a\right) \nonumber\\
& \leq C\sum_{n\geq 0}  d^n e^{-n I(a)}.
\end{align}
 Since $I(a)\to\infty$ as $a\to\infty$, we can choose $a$ such that the last sum converges. Then by Borel-Cantelli's lemma, $\IP$-almost surely for $n$ large enough we have $H_n(\mathbf x) \leq na$  for every path $\mathbf x$, so that
\[\limsup_{n\to\infty} \frac{1}{n}\log \DE_{x_0}\left[e^{\sum_{k=1}^n \b \omega(k,S_k)}\right] \leq \b a.\]
Since the support of $\omega$ is unbounded, we have $\Lambda(\b) \gg \b$ as $\b\to\infty$, therefore $\limsup n^{-1} \log W_n(x_0) < 0$ a.s.\ for $\b$ large enough. The proof is concluded via \eqref{eq:quenchedVerystrong}.
\end{proof}

 Proposition \ref{prop:veryStrongDis} provides
uniform bounds on the decay exponent.
\begin{proof}[Proof of Theorem \ref{th:VSDVG}]
  The proof parallels that of Theorem \ref{th:isoper}.
  By  Proposition  \ref{prop:veryStrongDis},
 the limsup in \eqref{eq:freeEnergyEstimateLacoin}
 does not depend on the starting point $x_0$.
Fix $x_0$ as in Assumption \ref{ass:covering} and let $n$ be large enough so that $(A_i)_{i\in I}$ is a covering of $B(x_0,nm)$.  We have
\begin{equation}
W_{nm}(x_0) \leq \sum_{i_1,\dots,i_m\in I} \hat W_{n,m}(i_1,\dots,i_m),
\end{equation}
where
\[\hat W_{n,m}(i_1,\dots,i_m) = \DE_{x_0}\left[e_{nm} \prod_{p=1}^m \mathbf{1}_{S_{np}\in A_{i_p}}\right].\]
By the formula $(a+b)^{\theta} \leq a^{\theta} + b^{\theta}$ for $a,b\geq 0$ and $\theta\in(0,1)$, we obtain that
\begin{equation} \label{eq:coarse}
\IE\left[W_{nm}^\theta\right] \leq \sum_{i_1,\dots,i_m\in I} \IE\left[\hat W_{n,m}(i_1,\dots,i_m)^\theta\right].
\end{equation}
For $i_1,\dots,i_m\in I$, let $J=J_0 \cup \dots \cup J_m$ with \[J_p = \{(k,x) \in (pn,(p+1)n]\times V: d(x,A_{i_p}) \leq C_1 n^{1/d_w}\}.\]
Set
$\delta_n = C_1^{-d_f/2} n^{-\frac{d_w+d_f}{2d_w}}$ . We have
\begin{align*}
&\IE\left[\hat W_{n,m}(i_1,\dots,i_m)^\theta\right] \\
&\leq  e^{\frac{\# J}{2}  \frac{\theta \delta_n^2 }{1-\theta}(1+o_n(1))} \DE_{x_0}\left[e^{-{\Lambda'(\beta) \delta_n (1+o_n(1))} \sum_{k=1}^{mn} \mathbf{1}_{(k,S_k) \in J}} \prod_{p=1}^m \mathbf{1}_{S_{np}\in A_{i_p}} \right]^{\theta}.
\end{align*}
Since $\inf \mu > 0$, by \eqref{eq:volume_growth},
for $n$ large enough the first factor on the right-hand side is bounded by
\begin{equation*} \label{eq:bound_changeOfmeas}
e^{mnC_V(C_1 n^{1/d_w})^{d_f} {\theta \delta_n^2 }/{(1-\theta)}} \leq e^{{m}\alpha},
\end{equation*}
with $\alpha = {\theta C_V }/{(1-\theta)}$.
We will now show that for all $m\geq 1$,
\begin{equation} \label{eq:expDecaySumLacoin}
\sum_{i_1,\dots,i_m\in I} \DE_{x_0}\left[e^{-{\Lambda'(\beta) \delta_n (1+o_n(1))} \sum_{k=1}^{mn} \mathbf{1}_{(k,S_k) \in J}} \prod_{p=1}^m \mathbf{1}_{S_{np}\in A_{i_p}} \right]^{\theta} \leq e^{-2m\alpha},
\end{equation}
which by \eqref{eq:coarse} will entail that
\begin{equation} \label{eq:expDecayLacoin}
\IE\left[W_{nm}^\theta\right] \leq e^{-m \alpha}.
\end{equation}
Using Markov's property, the summand in \eqref{eq:expDecaySumLacoin} is bounded by
\begin{align*} &\DE_{x_0}\left[e^{-{\frac{1}{2}\Lambda'(\beta) \delta_n } \sum_{k=1}^{m} \mathbf{1}_{S_k \in \tilde{J}_0}} \mathbf{1}_{S_{n}\in A_{i_1}} \right]^{\theta}\\
&\times \prod_{p=1}^{m-1} \sup_{x\in A_{i_p}} \DE_x\left[ e^{-{\frac{1}{2}\Lambda'(\beta) \delta_n } \sum_{k=1}^{n} \mathbf{1}_{S_k \in \tilde{J}_{i_p}}} \mathbf{1}_{S_{n}\in A_{i_{p+1}}}  \right]^{\theta},
\end{align*}
with $\tilde{J}_{i} = \{x\in V : d(x,A_i) \leq C_1 n^{1/d_w}\}$.
Hence \eqref{eq:expDecaySumLacoin} will follow once we show that 
\[{\sup_{j\in I}} \sum_{i} \sup_{x\in A_{j}}  \DE_x\left[ e^{-{\frac{1}{2}\Lambda'(\beta) \delta_n } \sum_{k=1}^{n} \mathbf{1}_{S_k \in \tilde{J}_j}} \mathbf{1}_{S_{n}\in A_{i}}  \right]^{\theta} \leq\varepsilon, \quad \varepsilon:=e^{-2 \alpha}.\]
We decompose the left hand side of the last display as
\begin{equation} \label{eq:2sums_lacoin}
\begin{aligned}
&\displaystyle\sum_{\substack{i\in I \\ d(A_i,A_j) \geq R n^{1/d_w}}}\sup_{x\in A_{j}}  \DE_x\left[ e^{-{\frac{1}{2}\Lambda'(\beta) \delta_n } \sum_{k=1}^{n} \mathbf{1}_{S_k \in \tilde{J}_j}} \mathbf{1}_{S_{n}\in A_{i}}  \right]^{\theta} \\
& \quad + \displaystyle\sum_{\substack{i\in I \\ d(A_i,A_j) < Rn^{1/d_w}}}\sup_{x\in A_{j}}  \DE_x\left[ e^{-{\frac{1}{2}\Lambda'(\beta) \delta_n } \sum_{k=1}^{n} \mathbf{1}_{S_k \in \tilde{J}_j}} \mathbf{1}_{S_{n}\in A_{i}}  \right]^{\theta}.
\end{aligned}
\end{equation}
The first sum is bounded by
\begin{align*}
&\sum_{k\geq R} \displaystyle\sum_{\substack{i\in I \\ d(A_i,A_j)\in[kn^{1/d_w},(k+1)n^{1/d_w})}} \sup_{x\in A_{j}} \DP_x\left(S_{n}\in A_{i}  \right)^{\theta}\\
& \leq C^\theta n^{-\theta d/2} \sum_{k\geq R} e^{-\theta \left(\frac{ k^{d_w}} c\right)^{ 1/{(d_w-1)}}} \displaystyle\sum_{\substack{i\in I \\ d(A_i,A_j) \in[kn^{1/d_w},(k+1)n^{1/d_w})}}  \mu(A_i)^\theta\\
& \leq(CC_V )^{\theta}  2^{\theta d} \sum_{k\geq R} e^{-\theta \left(\frac{ k^{d_w}} c\right)^{1/{(d_w-1)}}} C_G e^{ \left(\frac{ k^{d_w}} {c_2}\right)^{1/{(d_w-1)}}},
\end{align*}
where we have used \eqref{eq:gaussianBoundIsoper}, \eqref{eq:volume_growth} and \eqref{eq:unifDim} with the fact that $\mathrm{diam}(A_i) \leq n^{d/2}$. For $\theta$ close enough to $1$, the last sum can be made smaller than $\varepsilon /2$ by letting $R$ large enough (which we fix from now on).

The second sum in \eqref{eq:2sums_lacoin} is bounded from above by
\[
\# \{i\in I : d(A_i,A_j) < Rn^{1/d_w}\} \sup_{x\in A_{j}}  \DE_x\left[ e^{-{\frac{1}{2}\Lambda'(\beta) \delta_n } \sum_{k=1}^{n} \mathbf{1}_{S_k \in \tilde{J}_j}}  \right]^{\theta},
\]
where, by \eqref{eq:unifDim},
the first factor is bounded by some constant $C'=C'(R)$, and from
\eqref{eq-2709c} and the computation below it,
we find that for $C_1$ and $n$ large enough (in this order), the second factor is bounded  (uniformly in $j\in I$) by
\begin{equation} \label{eq:choiceC2}
\exp\left(-C_1^{-d_f/2} \frac{\Lambda'(\b)}{4}  n^{ \frac {2-d} 4}\right) + \frac{\varepsilon}{4C'}.
\end{equation}
Note that there exists $C>0$ such that for all $\b\leq 1$, we have $\Lambda'(\b) \geq C \b$.
We now choose $n$ to be any integer between $C_2 \beta^{-\frac{4}{2-d}}$ and $2C_2 \beta^{-\frac{4}{2-d}}$, with $C_2$ fixed big enough to make $n$ large enough and to ensure that the quantity in \eqref{eq:choiceC2} is less than $\varepsilon /(2C')$. This shows \eqref{eq:expDecaySumLacoin}.


Now, let $W_r(x,y) = \DE[e_r \mathbf 1_{S_r = y}]$ be the normalized point-to-point partition function. By Markov's property, we have
\[W_{s+r}(x_0) =\sum_{x\in V} W_s(x_0,x) W_r(x)\circ \eta_s\]
where $\eta_s$ is the shift of environment in time. We therefore get that
\begin{align*}
\IE\left[W_{s+r}(x_0)^\theta\right] & \leq \sum_{x\in V} \IE\left[W_s(x_0,x)^\theta\right] \IE\left[W_r(x)^\theta\right]
 \leq C_V s^{d_f} \IE\left[W_s(x_0)^\theta\right],
\end{align*}
since $\#\{x: \DP(S_s = x)>0\} \leq C_V s^{d_f}$ by \eqref{eq:volume_growth} and $\IE[W_r(x)] = 1$.

Hence, decomposing any $t> n$ into $t=nm_0+r$ with $r\in[0,n)$, we obtain along with \eqref{eq:expDecayLacoin} that
\[\IE\left[W_{t}^\theta\right]  \leq C_V t^{d_f} e^{-\alpha m_0}.
\]
Moreover, since
\[ t^{-1} \theta \IE \log W_{t} \leq t^{-1}\log  \IE[W_{t}^{\theta}],\]
we find that
\[\frac{\IE \log W_{t}}{t} \leq  \frac{\log C_V + d_f \log t}{t} - \frac{\alpha m_0}{(m_0+1)n},\]
so letting $t\to\infty$, we obtain that
\[\limsup_{t\to\infty} \frac{\IE \log W_{t}}{t} \leq - \frac{\alpha }{n}\leq -\frac{\alpha}{2C_2} \beta^{\frac 4 {2-d}},\]
with our choice of $n$ (see below \eqref{eq:choiceC2}). This gives \eqref{eq:freeEnergyEstimateLacoin}.
\end{proof}

\section{Proofs for Section \ref{subsec:L2regionIntro}  and
Theorems \ref{th:noL2rec} and \ref{th:transImpL2} -
the $L^2$-region} \label{sec:L2region}
\begin{proof}[Proof of Proposition \ref{prop:L2}]
From the definitions it follows that
the condition of $L^2$-boundedness in \eqref{eq:L2conditionIntro} reduces
to a condition on two independent copies of the random walk:
\begin{equation} \label{eq:L2moment}
\sup_n \IE[W_n(\b,x)^2] = \DE_{x,x}^{\otimes 2} \left[e^{\Lambda_{2}(\b)\sum_{i=1}^\infty \mathbf{1}_{S_i={S}'_i}}\right],
\end{equation}
where $\Lambda_2(\b)$ defined in \eqref{eq-lambda2}
is non-decreasing in $\b$.

We claim that finiteness of $\sup_n \IE[W_n(x)^2]$ does not depend on $x$.
Indeed, by Markov's property, we have for all $x,y\in V$,
\[ \DE_{x,x}^{\otimes 2}\left[e^{\Lambda_{2}(\b)\sum_{i=1}^\infty \mathbf{1}_{S_i={S}'_i}}\right] \geq \DP_{x,x}^{\otimes 2}(\tau_{(y,y)} < \infty) \DE_{y,y}^{\otimes 2}\left[e^{\Lambda_{2}(\b)\sum_{i=1}^\infty \mathbf{1}_{S_i={S}'_i}}\right],
\]
where $\tau_{(y,y)}=\inf_n \left\{n\geq 0: S_n=y, {S}'_n=y\right\}$ with $\DP_{x,x}^{\otimes 2}(\tau_{y,y} < \infty) > 0$ by irreducibily of the walk; this proves our claim.

Existence of the critical parameter $\b_2$ in Proposition \ref{prop:L2} then comes from \eqref{eq:L2moment}.
\end{proof}



\begin{proof}[Proof of Theorem \ref{th:diffL2}]

We follow the lines of Section 3.3 in \cite{CStFlour}.
Let $F$ be a test function. Since $W_n(x,\b)\to W_\infty(x,\b)$ with $W_\infty(x,\b) >0$ a.s, it is enough to show that
\[A_n:= \IE\left(\DE_{x}\left[e_n F(|S_n|/a_n) \right] - W_n(x,\b)\DE[F(X)] \right)^2 \to 0.\]
With $S,S'$ independent copies of $S$,
let $\overline{F}(x) = F(x) - \DE[F(X)]$ and $N_n = \sum_{k=1}^n \mathbf{1}_{S_k=S_k'}$ with $N=\lim_{n\to\infty} N_n$.
We have, with $\Lambda_2=\Lambda_2(\b)$ as in \eqref{eq-lambda2},
\begin{equation*}
A_n  = \DE_{x}^{\otimes 2}\left[e^{\Lambda_2 N_n} \overline{F}(|S_n|/a_n) \overline{F}(|{S}'_n|/a_n) \right],
\end{equation*}
so since the above integrand is bounded by $4\Vert F\Vert_\infty^2 e^{\Lambda_2 N}\in L^1$, it is enough to show that as $n\to\infty$,
\begin{equation} \label{eq:CVlaw3terms}
(N_n,|S_n|/a_n, |S_n'|/a_n)\cvlaw (N,X_1,X_2).
\end{equation}
where $N,X_1,X_2$ are independent and $X_i \eqlaw X$ for $i=1,2$.
Let $F_1,F_2,G$ be bounded Lipschitz functions
. Consider:
\begin{align*}
B_n & := \DE_{x}^{\otimes 2}\left[G(N_n) {F_1}(|S_n|/a_n) {F_2}(|{S}_n'|/a_n) \right]\\
& = \DE_{x}^{\otimes 2}\left[G(N_\ell) {F_1}(|S_n|/a_n) {F_2}(|{S}_n'|/a_n) \mathbf{1}_{N_n=N_\ell} \right] + \varepsilon_{n,\ell}^1\\
& = \DE_{x}^{\otimes 2}\left[G(N_\ell) \DE^{\otimes 2}_{S_\ell,S_\ell'} \left[ {F_1}(|S_{n-\ell}|/a_n){F_2}(|{S}'_{n-\ell}|/a_n)\right] \mathbf{1}_{N_n=N_\ell} \right] + \varepsilon_{n,\ell}^2\\
& = \DE_{x}^{\otimes 2}\left[G(N_\ell) \DE_{S_\ell} \left[ {F_1}(|S_{n-\ell}|/a_{n})\right] \DE_{S'_\ell}\left[{F_2}(|{S}'_{n-\ell}|/a_{n})\right] \right] + \varepsilon_{n,\ell}^3,
\end{align*}
where the $\varepsilon^i_{n,\ell}$ are defined implicitly.
We have:
\[\max\left\{\varepsilon^1_{n,\ell}, \varepsilon^3_{n,\ell}-\sum_{i=1,2} \varepsilon^i_{n,\ell}\right\} \leq \Vert G \Vert_\infty \Vert F_1 \Vert_\infty \Vert F_2 \Vert_\infty \DP_x^{\otimes 2}(N_n\neq N_\ell)\to 0,\]
uniformly in $n\geq \ell$ as $\ell\to\infty$ since $N_n\nearrow N$. Furthermore, for fixed $\ell$, we obtain that $\varepsilon^2_{n,\ell} - \varepsilon^1_{n,\ell}  \to 0$ as $n\to\infty$ by Markov property using that
\begin{align*}
|\DE_x[F_1(|S_{n}|/a_n)] - \DE_x^{\otimes 2}[F_1(d(S_n,S_\ell)/a_n)]| & \leq \Vert F_1 \Vert_{Lip} \DE_x[d(S_\ell,x)/a_n]\\
& \leq \Vert F_1 \Vert_{Lip}\,  \ell\,  a_n^{-1},
\end{align*}
with $a_n^{-1} \to 0$ as $n\to\infty$, where $\Vert F_1 \Vert_{Lip}$ the Lipschitz constant of $F_1$. Now for fixed $\ell$, we have by hypothesis that $|{S}_{n-\ell}|/a_{n} \cvlaw X$ as $n\to\infty$, therefore letting first $n\to\infty$ and then $\ell\to\infty$, we obtain that $B_n \to \DE_x^{\otimes 2}[G(N)] \DE[F_1(X)] \DE[F_2(X)]$, which entails \eqref{eq:CVlaw3terms}.
\end{proof}
\begin{proof}[Proof of Theorem \ref{th:noL2rec}]
Let $(S_k)$ be recurrent and reversible with respect to a measure $(\pi(x))_{x\in V}$ which is bounded away from zero. We will show that $\b_2 = 0$.
We first observe that $(S_{2k})$ is recurrent as well. Indeed, starting from any vertex $x$, either $(S_{2k})$ or $(S_{2k+1})$ visits $x$ infinitely many times. If $x$ has even period, then it has to be $(S_{2k})$; if it has odd period, then both visit $x$ infinitely often by irreducibility of the walk.

Now, with $S,S'$ denoting independent copies of $S$, let
\[N_\infty(S,S')=\sum_{i=1}^\infty \mathbf{1}_{S_i=S'_i}.\]
Since the reversible measure satisfies $\pi(x)\DP_x(S_n=y) = \DP_y(S_n=x) \pi(y)$, we have:
\begin{align} \DE^{\otimes 2}_{x,x} N_\infty(S,S') & = \sum_{n\geq 0} \sum_{y\in V} \DP_x(S_n = y)^2 \nonumber\\
& = \sum_{n\geq 0} \sum_{y\in V} \DP_x(S_n = y) \frac{\pi(y)}{\pi(x)} \DP_y(S_n = x) \label{eq:NinfGreenRev} \\
&  \geq \frac{\inf_{y} \pi(y)}{\pi(x)} \sum_{n\geq 0} \DP_x(S_{2n} = x) = \infty, \nonumber
\end{align}
where the last equality holds since $(S_{2k})_{k\geq 0}$ is recurrent. Hence the RHS of \eqref{eq:L2moment} is infinite when $\b > 0$ by the lower bound $e^{x} \geq x$.
\end{proof}


\begin{proof}[Proof of Theorem \ref{th:transImpL2}]
Recall the definition of the Green function of $S$ in \eqref{eq:defGreenFunction}. Let $(S_k)$ be transient and reversible with respect to a measure $(\pi(x))_{x\in V}$ which is bounded from above and such that $\sup_{x\in V} G(x,x)/\pi(x) < \infty$. We show that $\b_2>0$.

By Khas'minskii's lemma \cite[P.8, Lemma 2]{Sznitman}, there is an $L^2$-region whenever
\begin{equation}\label{eq:Khas}
\sup_{x\in V} \DE_{x,x}^{\otimes 2} \left[ N_\infty(S,S') \right] < \infty,
\end{equation}
and we have by \eqref{eq:NinfGreenRev},
\begin{equation*} \DE^{\otimes 2}_{x,x} N_\infty(S,S')   \leq \frac{C}{\pi(x)} \sum_{k\geq 0} \DP_x(S_{2k} = x)
 \leq  C\frac{G(x,x)}{\pi(x)},
\end{equation*}
from which the statement of the theorem follows.
\end{proof}

\subsection{Criteria for $\b_2=0$ and proof of Theorem
\ref{th:ConditionForNoL2}}
The following expression for  $\IE(W_\infty^2)$ will be useful in this section.
Recall $\Lambda_2=\Lambda_2(\b)$, see
\eqref{eq-lambda2}.
\begin{proposition} \label{prop:L2viaChaos} The following identity holds:
\begin{equation*}
  \lim_{n\to\infty} \IE\left[ W_n^2(x,\b) \right] = 1 + \sum_{k=1}^\infty \left(e^{\Lambda_2} -1 \right)^k \sum_{1\leq n_1<\dots < n_k} \DE_x^{\otimes 2} \left[ \mathbf{1}_{S_{n_1}=S'_{n_1},\dots, S_{n_k} = S'_{n_{k}}}  \right].
\end{equation*}
\end{proposition}
\begin{proof}  We have
\[\IE\left[ W_n^2 (x)\right] = \DE^{\otimes 2}_x\left[ e^{\Lambda_2\sum_{i=1}^n  \mathbf{1}_{S_{i}= S'_i}}\right] = \DE^{\otimes 2}_x\left[\prod_{i=1}^n \left(1+(e^{\Lambda_2}-1) \mathbf{1}_{S_{i}= S'_i} \right) \right].\]
Now expand the product and take $n\to\infty$,
using monotone convergence.
\end{proof}



\begin{proof}[Proof of Theorem \ref{th:ConditionForNoL2}]
By Proposition \eqref{prop:L2viaChaos}, it is enough to show that for some $x_0\in V$ and for all $\delta>0$, there exists a constant $C>0$ such that for all $k\geq 1$,
\begin{equation}
  \label{eq-2609a}
  \sum_{1\leq n_1<\dots < n_{k}} \DE_{x_0}^{\otimes 2} \left[ \mathbf{1}_{S_{n_1}=S'_{n_1},\dots, S_{n_{k}} = S'_{n_{k}}}  \right] \geq C \delta^{-k}.
\end{equation}
We fix an arbitrary
$x_0\in V$ and $\delta>0$. By the assumption of the theorem,
there exists  $x^\star\in V$ such that
\[\sum_{m\geq 0} \DP_{x^\star}(S_m=x^\star)^2 \geq \delta^{-1}. \]
Therefore,
\begin{align*}
&\sum_{1\leq n_1<\dots < n_{k}} \DE_{x_0}^{\otimes 2} \left[ \mathbf{1}_{S_{n_1}=S'_{n_1},\dots, S_{n_{k}} = S'_{n_{k}}}  \right]\\
& = \sum_{m_1,\dots,m_k} \sum_{x_1,\dots,x_{k}\in V} \prod_{i=0}^{k-1} \DP_{x_{i}}\left(S_{m_i}= x_{i+1}\right)^2\\
& \geq \left( \sum_{m_1} \DP_{x_0}(S_{m_1} = x^\star)^2\right) \left( \sum_{m} \DP_{x^{\star}}(S_{m} = x^\star)^2 \right)^{k-1}
 \geq C \delta^{-k},
\end{align*}
where $C>0$  by the irreducibility of $(S)$.
This proves \eqref{eq-2609a}.
\end{proof}

We describe an application of  Theorem \ref{th:ConditionForNoL2} to
the construction of a family of transient graphs satisfying $\b_2=0$.
\begin{definition}
  \label{def-pipe}
A \emph{pipe} in $G$ is
a chain of vertices $v_1,\ldots,v_L$ satisfying
\begin{enumerate}
  \item $(v_i,v_{i+1})\in E$.
  \item  For $i=2,\ldots,L-1$, the degree of $v_i$ is $2$.
\end{enumerate}
\end{definition}
\begin{proposition} \label{cor:NoL2forLongPipes}
Let $G$ contain arbitrarily long pipes. Then
$\b_2=0$ for the $G$/SRW polymer.
\end{proposition}
\begin{proof}
Consider a pipe of length $L$ and let $x$ be its center. Then,
 $p_k(x,x) = p_k^\mathbb{Z}(x,x)$ for $k< L/2$, where $ p_k^\mathbb{Z}(x,y)$ is the transition probability for simple random walk on $\mathbb Z$.
 Hence $\sum_{k} p_k(x,x)^2 \geq  \sum_{k=1}^{L/2} p_k^\mathbb{Z}(x,x)^2 \sim \log L \to \infty $ as $L\to\infty$, which is condition \eqref{eq-bloc58c}.
 An application of Theorem \ref{th:ConditionForNoL2} completes the proof.
\end{proof}


\noindent
The condition in Theorem \ref{th:ConditionForNoL2} can be relaxed to the following, writing $\tau_S(x,L)$ for the hitting time of  a
ball of radius $L$ around $x$. The proof is identical to that of Theorem \ref{th:ConditionForNoL2}
and is therefore omitted.
\begin{theorem} \label{th:ConditionForNoL2bis}
Assume that
\begin{equation}
\label{bloc58}\sup_{L>0}\sup_{x\in V} \inf_{y\in B(x,L/2)} \DE^{\otimes 2}_{y}\sum_{k\geq 0} \mathbf{1}_{S_{k} =S'_{k}, S_k\in B(x,L/2), \tau_S(x,L)>k, \tau_{S'}(x,L)>k} = \infty.
\end{equation}
Then, $\b_2=0$.
\end{theorem}
\subsection{$\b_2=0$ for the lattice supercritical percolation cluster}
In this section, we consider $G$ to be a   supercritical
percolation cluster on $\Z^d$, see Section \ref{subsec-perco} for definitions.
\begin{theorem}
\label{th:noL2perco}
$\b_2=0$ a.s. for SRW on the supercritical percolation cluster of $\mathbb Z^d$,
$d\geq 2$.
\end{theorem}
Theorem \ref{th:noL2perco} follows at once from Corollary \ref{cor:NoL2forLongPipes} and the next lemma.
Recall the notion of pipe, see Definition \ref{def-pipe}.
\begin{lemma} For all $L\geq 1$ and $d\geq 2$,
there exist a.s.
infinitely many pipes of length $L$ in the supercritical percolation cluster.
\end{lemma}
\begin{proof}
Partition $\Z^d$ to boxes of side $L+1$.
Fix such a box $B$ and let $\mathcal{C}$ denote the (unique)
infinite cluster. Let $\mathcal E$ denote all edges that connect two vertices in the boundary of $B$.  Define the events
$$A_1=\{B\cap \mathcal{C}\neq \emptyset\},\quad
A_2=\{\mbox{\rm All edges in $\mathcal E$ are open}\}.$$
Note that $A_1$ and $A_2$ are increasing functions. Hence, by FKG
and $p>p_c$,
$$P(A_1\cap A_2)\geq P(A_1)\cdot P(A_2)>0.$$
On the other hand, let ${\mathcal P}_L$ denote the event that there exists in $B$ a pipe $(v_1,\ldots,v_L)$ of length $L$ with $v_1$ belonging to  the boundary of $B$. Then
since on $A_2$, $A_1$ only depends on the configuration outside $B$,
we have that
$$P( {\mathcal P}_L|A_1\cap A_2)=P({\mathcal P}_L|A_2)>0.$$
Combining the above, we get
$$P({\mathcal P}_L\cap A_1)\geq P({\mathcal P}_L \cap A_1\cap A_2)>0.$$
Thus, with positive probability, there exists a pipe of length $L$ in the supercritical percolation cluster. By ergodicity, this implies the lemma.
\end{proof}

\subsection{A transient graph with $0=\b_2<\b_c$}
\label{subsec:transGraphWeakButNoL2}
Recall Definition \ref{def-pipe}.
Consider the graph $G_{\ref{th:ZdwithPipes}}$
that is obtained by glueing to $\mathbb Z^d$,
on the $k$th vertex of the line
\[\mathcal D :=\{(k,0,\dots,0)\in\mathbb Z^d,k\geq 1\},\]
 pipes of length $k$.
 See Figure \ref{fig:ZDPIPES} for an illustration
 \begin{figure}
  \includegraphics[scale=0.3]{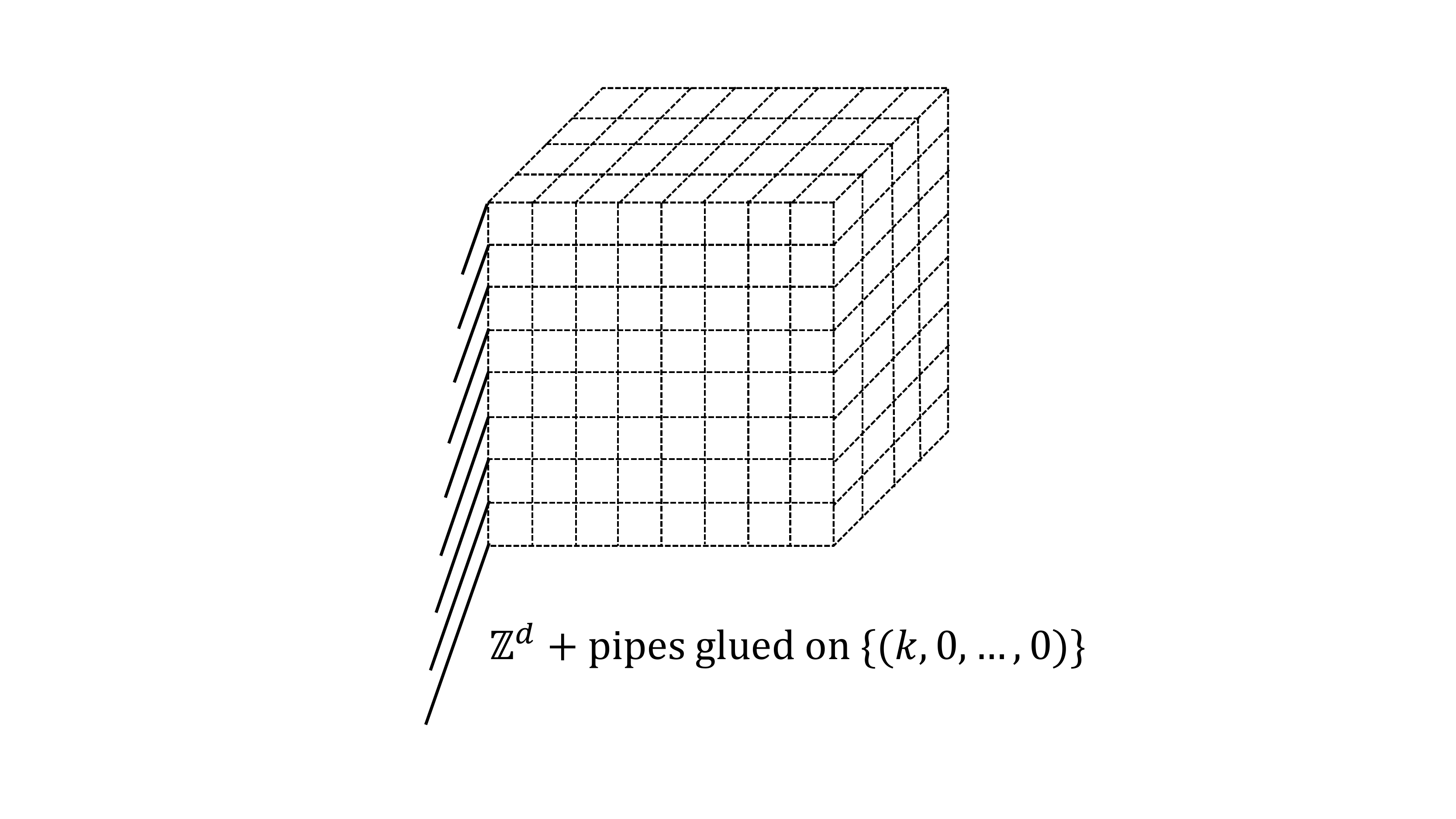}
  \caption{The graph  $G_{\ref{th:ZdwithPipes}}$, for which 
   $0=\b_2<\b_c$. }
  \label{fig:ZDPIPES}
\end{figure}

\begin{theorem} \label{th:ZdwithPipes}
The polymer on
$G_{\ref{th:ZdwithPipes}}$/SRW with $d\geq 4$ 
\end{theorem}
\begin{proof}
  In view of Corollary \ref{cor:NoL2forLongPipes}, for the first assertion
  it suffices to prove that
  $\b_c>0$. Denote by $A=A(S)$ the event $A=\{S \text{ never enters } \mathcal D\}$. Then, when $d\geq 4$,
\[\DP_0(A)=\DP^{\mathbb Z^d}_0(A)>0,\]
where $(\DP^{\mathbb Z^d},S)$ is the SRW on $\mathbb Z^d$.

We now introduce the martingale
$W_n(0,A)=\DE_0[e_n \mathbf{1}_A]$. We will
prove that, for some $\b>0$,  $W_n(0,A)$  is uniformly
bounded in $L^2$. This will imply
that for such $\b$,   $\IE[W_\infty(0,A)] = \DP_0(A) > 0$,
and since $W_\infty(0) \geq W_\infty(0,A)$ a.s., we further obtain that
$\IE[W_\infty(0)] > \IE[W_\infty(0,A)] > 0$ and hence, by
Proposition \ref{prop:StrongWeakDisorderIntro}, that $\b_c > 0$.

Turning to the $L^2$ estimate, as in the proof of Proposition \ref{prop:L2viaChaos} we have,  letting $\mu_2=\mu_2(\b) = e^{\Lambda_2(\b)} -1$,
\begin{align*}
&\IE\left[ W_\infty(0,A)^2 \right]\\
& = 1 + \sum_{k=1}^\infty \mu_2^k \sum_{1\leq n_1<\dots < n_k} \DE_0^{\otimes 2} \left[ \mathbf{1}_{S_{n_1}=S'_{n_1},\dots, S_{n_k} = S'_{n_{k}}} \mathbf{1}_{A(S)} \mathbf{1}_{A(S')}  \right],
\end{align*}
with
\begin{align*}
&\sum_{1\leq n_1<\dots < n_k} \DE_0^{\otimes 2} \left[ \mathbf{1}_{S_{n_1}=S'_{n_1},\dots, S_{n_k} = S'_{n_{k}}} \mathbf{1}_{A(S)} \mathbf{1}_{A(S')}  \right]\\
& = \sum_{1\leq n_1<\dots < n_k} \left(\DE_0^{\mathbb Z^d}\right)^{\otimes 2} \left[ \mathbf{1}_{S_{n_1}=S'_{n_1},\dots, S_{n_k} = S'_{n_{k}}} \mathbf{1}_{A(S)} \mathbf{1}_{A(S')} \right]\\
& \leq \sum_{1\leq n_1<\dots < n_k} \left(\DE_0^{\mathbb Z^d}\right)^{\otimes 2} \left[ \mathbf{1}_{S_{n_1}=S'_{n_1},\dots, S_{n_k} = S'_{n_{k}}}\right].
\end{align*}
Therefore, we obtain by summing that, for $\b$ small,
since $d\geq 4$,
\[\IE\left[ W_\infty(0,A)^2 \right] \leq \IE\left[ \left(W_\infty^{\mathbb{Z}^d}(0)\right)^2 \right] <\infty,\]
where 
 $W_\infty^{\mathbb{Z}^d}(0)$ denotes $W_\infty( 0)$ for the polymer on $\Z^d$/SRW. 
 
 We turn to the last assertion of the theorem. By Corollary \ref{cor:Liouville}, it is enough to show that $G$ satisfies Liouville's property. Let $h$ be a harmonic function on $G$. By harmonicity, $h=h_k$ on the $k$-th pipe. Thus,
$h$ restricted to $\mathbb Z^d$ is again harmonic. Since $\mathbb Z^d$ satisfies the Liouville property,
it follows that so does $G$.
\end{proof}

\section{Polymers on tree structures}
\label{sec-5}
Throughout this section
we take $S$ to be the $\lambda$-biased
random walk on either the
Galton-Watson trees
with offspring distribution $(p_k)_{k\geq 0}$ such that $m = \sum k p_k >1$
(conditioned on non-extinction if $p_0>0$), or on the canopy tree,
see
Sections \ref{sec:GWtree} and \ref{sec:canopyTree}
for definitions.
We write
$D_n=\{v\in V: d(v,o)=n\}$ for vertices at level $n$ of the tree.
In the case of the Galton-Watson tree,
expectations and probabilities with respect to the randomness of the  tree
will be denoted by $\DE_\mu$ and $\mu$ respectively.
When $p_0>0$, we write $q=\mu(\text{extinction})$ and
$\mu_{0}(\cdot) = \mu(\cdot |\text{ non-extinction})$.

\subsection{The positive recurrent Galton-Watson tree ($\mathbf{\lambda > m}$)} \label{sec:posRecGW}
Let $(\mathcal T,(S_k))$ be the walk on the Galton-Watson tree $\mathcal T$ with parameters $\lambda > m>1$, as defined in Section \ref{sec:GWtree}.
\begin{theorem} \label{th:posRecGW}
Assume $\lambda>m$. Then, almost surely on the realization of $\mathcal T$,
strong disorder always holds ($\b_c=0$). Moreover,
\begin{enumerate}[label=(\roman*)]
\item If the tree is $m-1$ regular (i.e. ${p_m=1}$ for some integer $m$), then \textbf{very} strong disorder always
  holds, i.e. $\bar{\b_c}(\mathcal T) = 0$, a.s.
\item More generally, if $\sup \{{d}:p_d >0\} < \lambda$, then $\mu_0$-a.s,
  $\bar{\b_c}(\mathcal T) = 0$.
\item If $\DE_\mu[|D_1|\log^+ |D_1|]<\infty$, $p_0>0$  and
  there exists some $d>\lambda$ such that $p_d>0$, then
  $\bar{\b_c}(\mathcal T) > 0$, $\mu_0$-a.s. 
\end{enumerate}
\end{theorem}
\begin{remark}
Point (iii) shows in particular that there are positive recurrent walks such that very strong disorder does not always hold.
\end{remark}
\begin{proof}
The fact that $\b_c=0$ comes from Theorem \ref{th:strongDesForStrongRec}.
In order to see (i), let $X_n=|S_n|$.
Then $(X_n)$ is a random walk on $\mathbb Z_+$ with a bias
towards $0$ which is constant on each point of $\mathbb Z_+$.
Therefore, the return time to $0$ of $(S_n)$ admit exponential moments
and by \cite{CaGuHuMe04}, very strong disorder always holds.
Point (ii) holds since the problem can be reduced to a random walk $(Y_n)$ on half a line with inhomogeneous bounded bias towards the root. By a standard coupling argument, the walk $(X_n)$ will stochasticaly dominate the walk $(Y_n)$, which then implies, following from (i), exponential moments for the return time.

We next prove (iii). Letting $Z = \lim_{n} m^{-n} |D_n|$, it
follows from the Kesten-Stigum theorem \cite{KestenStigum66} that
the condition  $\DE_\mu[|D_1|\log^+ |D_1|]<\infty$
implies that $\mu_{0}(Z > 0)=1$.
Thus, for $\varepsilon >0$ small  and any $c>0$,
\begin{equation} \label{eq:growthTree}
  \text{if}\; A_n=\{|D_{\lfloor (\log n)^c \rfloor}| \geq (m-\varepsilon)^{\lfloor (\log n)^c \rfloor}\}, \text { then}\; \mu_0(\cup_{m\geq 1} \cap_{n\geq m}
  A_n)=1.
\end{equation}
In what follows, we write $l=l_n=\lfloor (\log n)^c\rfloor$.
Now fix any $d>\lambda$ such that $p_d>0$ and denote by $\mathcal R_n$ the $(d+1)$-regular tree of depth $L = \lfloor \log \log n \rfloor$. With a slight abuse of notation we also let
\[B_n=\{\text{There is at least one } \mathcal R_n \text{ rooted at one of the vertices of } D_{l}\}.\]
On the event $A_n$, there are at least $(m-\varepsilon)^l$
vertices in $D_{l}$
that may independently spawn $\mathcal R_n$ with probability $p_d^{Q} \, p_0^{d^{L}}$, where $Q=1+d+\dots + d^{L -1}$. Therefore,
\begin{align*}
  \sum_{n\geq 0} \mu_0(A_n,B_n^c) & \leq (1-q)^{-1}
  \sum_{n\geq 0} \left(1- p_d^{Q} \, p_0^{d^{L}}\right)^{(m-\varepsilon)^l
 }<\infty,
\end{align*}
where the finiteness of the last sum comes from
taking $c=c(d)$ large enough. By the
Borel-Cantelli lemma and \eqref{eq:growthTree}, we thus obtain that
\begin{equation} \label{eq:ThicktreeEv}
  \mu_0({\cup_{m\geq 1}} \cap_{n\geq m}B_n) =1.
\end{equation}


We now fix a realization of the infinite tree $\mathcal T$
and $n$ large so that
$B_n$  holds, and pick an $\mathcal R_n$ corresponding to that event.  (Such an $n$ exists $\mu_0$-almost surely by \eqref{eq:ThicktreeEv}). 
Introduce the event
\begin{align*}
  F_n(S) &:= \{ S_l\in \mathcal R_n,  \\
  &
|S_{k+1}|=|S_k|+1\; \text{for}\; k=1,\ldots l+L,
|S_k|>l\; \text{for}\;  k=l+L+1,\ldots, n\}.
\end{align*}
In words, the event $F_n(S)$ means that the random walk goes directly
to the bottom of one of the $\mathcal R_n$, and does not
reach the root of that $\mathcal R_n$ before time $n$.
What we show next is that the event $F_n(S)$
has  sub-exponential probability, namely
that there exists positive constants $c_1,c_2$
that depend only on $\lambda$ and $d$, such that for all $n$ large enough,
\begin{equation} \label{eq:sub-expproba}
\DP_{o}(F_n) \geq \exp\left\{-c_1\frac{n}{(\log n)^{c_2}}\right\}.
\end{equation}
To prove \eqref{eq:sub-expproba},
we first observe that on $B_n$,
the event
$\{S_l\in \mathcal R_n,  |S_{k+1}|=|S_k|+1\; \text{for}\; k=1,\ldots l+L\}$
has  probability bounded from below by $Ce^{-(l+L)}$.
Once at the bottom, the probability of reaching the top of $\mathcal R_n$
before returning  to the bottom of $\mathcal R_n$ equals
the probability of reaching $L$ before reaching $1$ for
a SRW on $[1,L]\cap \mathbb N$ started at $2$,
with probability $\lambda/(\lambda+d)$ to go right and $d/ (\lambda+d)$
to go left. Since $d>\lambda$, this probability is
equivalent as $L\to\infty$ to $c_3 (\lambda/d)^{L}$,
with $c_3=c_3(\lambda,d)$. Therefore, the probability starting from the bottom
of $\mathcal R_n$ not to reach the root of $\mathcal R_n$ at all before time $n$ is bounded from below by (recall that $L=\lfloor \log \log n \rfloor$)
\[\left(1-c_2 (\lambda/ d)^{L} \right)^n \sim e^{-c_1\frac{n} {(\log n)^{\log
(d/\lambda)}}},\]
for some constant $c_1=c_1(d,\lambda)$, where we recall that
$L=\lfloor \log \log n \rfloor$. Combining these estimates with the Markov property
leads to \eqref{eq:sub-expproba}.

We can now turn to the conclusion of the proof. Define
$Y_n:=\DE_{o}[e_n F_n]$. We will show below that for $\b>0$ small enough,
\begin{equation} \label{eq:bound2ndMomYn}
\exists \b>0,\quad \IE Y_n^2 \leq e^{\Lambda_2(\b)(l+L)}.
\end{equation}
Assuming \eqref{eq:bound2ndMomYn}, we have
by the Paley-Zygmund inequality that
\begin{equation} \label{eq:paley-zygForYn}
\IP\left(Y_n \geq \frac{1}{2} \IE[Y_n] \right) \geq \frac{1}{4} \frac{\IE[Y_n]^2}{\IE Y_n^2}\geq  \frac 1 4
\exp\left(-3c_1\frac{n}{(\log n)^{c_2}}\right),
\end{equation}
for $n$ large enough, where we used \eqref{eq:sub-expproba} and that
$\IE[Y_n] = \DP_{o}(F_n)$.
Since $W_n \geq Y_n$, we further obtain from
\eqref{eq:paley-zygForYn}
that there exists
a positive sequence $\alpha_n$ satisfying $\alpha_n = o(n)$
as $n\to\infty$, such that
\begin{equation} \label{eq:sub-expWn}
\IP(\log W_n \geq -\alpha_n) \geq C e^{-c' \alpha_n},
\end{equation}
for some positive constants $C$ and $c'$. On the other hand,
if very strong disorder holds for
$\b>0$, that is if
 $\IE[\log W_n] \leq -\varepsilon n$ for $n$ large enough and some
 $\varepsilon \in(0,1)$,
 then by the concentration inequality \eqref{eq:ConcentrInequality},
we obtain that
\[\IP\left(\log W_n \geq - \frac{\varepsilon}{2}n\right) \leq  e^{-C'(\b) n \varepsilon^2},
\]
which cannot hold in the same time as \eqref{eq:sub-expWn}.
and hence very strong disorder does not hold for such $\b$.

We now come back to the proof of \eqref{eq:bound2ndMomYn}. We have,
\[\IE Y_n^2 = \DE_{o,o}^{\otimes 2} \left[e^{\Lambda_2(\b) \sum_{k=1}^n \mathbf{1}_{S_k = S'_k}} F_n(S) F_n(S')\right].
\]
Since on $F_n$, both walk go directly to the bottom of some $\mathcal R_n$ as above (this takes $(\log n)^c$ steps), we have that for $x_0$ which is a leaf of  $\mathcal R_n$ that
\[\IE Y_n^2 \leq e^{\lambda_2(\b)(l+L)} \DE^{\otimes 2}_{x_0,x_0} \left[e^{\lambda_2 \sum_{k=1}^n \mathbf{1}_{S_k = {S}'_k}} \mathbf{1}_{S \text{ and }  S' \text{ do not exit } \mathcal R_n \text{ before time } n}\right].
\]
Then, observe that the number of intersections
of the walks before they leave $\mathcal R_n$ is stochastically dominated
by the total number of
intersections that would occur in any
infinite tree having $\mathcal R_n$ attached to some vertex.
In particular, \eqref{eq:bound2ndMomYn} holds
if we can find such an infinite tree $T$ for which $\b_2>0$.
We will choose $T$
to be the canopy tree  with parameters $m=d$ and $\lambda$, see
Section \ref{sec:canopyTree} for definitions.
The required finiteness of $\beta_2$ for the canopy tree
now follows from
Theorem \ref{th:L2regionCanopy} below.
\end{proof}
\subsection{The transient Galton-Watson tree ($\mathbf{\lambda < m}$)}
\label{sec:transGW}
Let $(\mathcal T,(S_k))$ be the walk on the Galton-Watson tree $\mathcal T$ with parameters $\lambda < m$, as defined in Section \ref{sec:GWtree}.
\begin{theorem} \label{th:transientGW}
Assume that $m>1$, $\lambda<m$, and $p_0=0$.
Then,
almost surely on the realization of $\mathcal T$, $\b_c>0$. If further
$p_1>0$, then
$\b_2=0$.
\end{theorem}
\noindent
We remark that the particular case $\lambda=1$, which corresponds to the SRW,
is covered by Theorem \ref{th:transientGW}.
\begin{proof}
  That $\b_2=0$ when $p_1>0$ follows from Proposition \ref{cor:NoL2forLongPipes}
  upon observing that there exist (due to $p_1>0$) arbitrarily long pipes.
  We thus only need to prove that $\b_c>0$. Throughout the proof, we write $GW$ for the law of $\mathcal{T}$. Recall that for $v\in V$,
  $\tau_v=\inf\{k>0: S_k=v\}$. A preliminary step in the following lemma.
  \begin{lemma}\label{lem:proba_reaching_GW}
    There exists $c\in(0,\infty)$ such that $GW$-a.s.,
    there exists $N=N(\mathcal T)>0$ such that for all $n\geq N$,
\begin{equation} \sup_{v\in D_n} \DP(\tau_v < \infty) \leq  e^{-cn}.
\end{equation}
\end{lemma}
\begin{proof} Let $\delta>0$.
For $w\in V$ and $j$ a descendant of $w$, define
\[A(w,j)=\mathbf{1}_{\DP_w(\tau_j=\infty)>\delta}.\]
For $v\in D_n$, let $[x_0=o,\dots,x_n=v]$
be the unique simple
path connecting the root to $v$.
It follows from \cite[Lemma 2.2]{dembo2002large}
that for $\delta>0$ small enough, there exist
$\alpha>0$ and $N_0(\mathcal T)$ such that for $n \geq N_0(\mathcal T)$,
\begin{equation*}\label{eq:unifProbaRay}
GW-a.s,\quad \inf_{v\in D_n} n^{-1} \sum_{i=1}^n A(x_{i-1},x_i) \geq \alpha.
\end{equation*}
Fix $v\in D_n$. By the Markov property, we have that for $n\geq
N_0(\mathcal T)$,
\[\DP(\tau_v<\infty) = \prod_{i=1}^n \DP_{x_{i-1}}(\tau_{x_{i}} <\infty)
\leq  (1-\delta)^{-\alpha n},\]
with $C=C(\tau)$ and where the last bound holds uniformly over $v\in D_n$.
\end{proof}
We also need the following lemma.
In the statement, $S,S'$ denote two independent copies of $(S_k)$ on $\mathcal{T}$. Let $\mathcal {F}_S$ denote the
$\sigma$-algebra generated by $S$. Also, let
\[\mathcal{N}=\{k\geq 0: S_k=S_k'\}, \quad \mathcal{V}=\{v\in V: S_k=v
\; \text{for some $k\in \mathcal{N}$}\}\]
denote the intersection times and locations of
$S$ and $S'$.
\begin{lemma}
  \label{lem:donotmeet}
  There exists $c_2\in(0,\infty)$ such that $GW$-a.s.,
  there exists $\ell_0=\ell_0(\mathcal{T},S)$ so that
\begin{equation}\label{eq:donotmeet}
  \text{for all $\ell>\ell_0$},\quad
  \DP^{\otimes 2}\left(\mathcal{V}\cap (\cup_{m\geq \ell} D_m)\neq \emptyset
  \middle| \mathcal{F}_S\right) \leq e^{-c_2\ell}, \quad {GW-}a.s.
\end{equation}
\end{lemma}
\begin{proof}
We say that $\ell$ is a \emph{regeneration level} for $S$ if there is a time $\sigma\geq 0$ such that $|S_k| \geq \ell$ for all $k\geq \sigma$.
It follows from \cite[Lemma 4.2]{dembo2002large}
that  $S$ possesses infinitely many regeneration levels,
whose successive differences are independen and admit exponential moments under
the annealed law $GW\times \DP$.
In particular, this
implies that, for some $c>0$,
\begin{equation} \label{eq:lemma4.2DemboAl}
\IE_{GW}\DP(\text{there is no regeneration level for } S \text{ in } [\ell/2,\ell])\leq e^{-c\ell}.
\end{equation}
Therefore, by the
Borel-Cantelli lemma, $GW\times \DP$ a.s.,
there is at least one regeneration level $\ell_0$
for $S$ in $[\ell/2,\ell]$ for
$\ell\geq \ell_1(\mathcal{T},S)$ large enough.
We continue the proof on that event.

Let $h\in D_\ell$ be the last vertex visited by $S$ at level $\ell/2$ before it regenerates at level $\ell_0$. We have, using Lemma \ref{lem:proba_reaching_GW}
at the second
inequality, that
\begin{align*}
  \DP^{\otimes 2}\left(\mathcal{V}\cap (\cup_{m\geq \ell} D_m)\neq \emptyset
  \middle| \mathcal{F}_S\right) &\leq
 \DP(\tau_{h}(S') < \infty)
 \leq  e^{-c\ell/2}.
\end{align*}
\end{proof}

We return to the proof of
Theorem \ref{th:transientGW}.
By Theorem \ref{th:Birkner},
it is enough to show that for some $\b>0$,
\begin{equation}
  \DE^{\otimes 2}\left[ e^{\Lambda_2(\b) |\mathcal{N}|} \middle| \mathcal{F}_S\right] < \infty, \quad GW\times \DP-a.s.
\end{equation}
It follows from \cite{LyonsPemantlePeresBiased} that
\begin{equation} \label{eq:speedGW} \frac{|S_k|}{k}\to v>0, \quad GW\times \DP-a.s.
\end{equation}
In particular, there is a random variable
$c_1=c_1(S)$ such that for all $j\geq 0$, $S$ does not spend more than
$c_1 v^{-1} k$ time above level $k$. Setting
\[L = \inf\left\{\ell \geq 0 : \mathcal{V}\cap (\cup_{m\geq \ell} D_m)= \emptyset  \right\},\]
we obtain from the
Borel-Cantelli lemma and \eqref{eq:donotmeet} that $L<\infty$,
$GW\times \DP^{\otimes 2}$-a.s., and therefore
%
\begin{align}
  \DE^{\otimes 2}\left[ e^{\Lambda_2 |\mathcal{N}|}
  \middle| \mathcal{F}_S\right] & = \sum_{k\geq 0}
  \DE^{\otimes 2}\left[ e^{\Lambda_2 |\mathcal{N}|}
  \mathbf{1}_{L=k} \middle| \mathcal{F}_S\right]
  \leq  \sum_{k\geq 0} e^{c_1 \Lambda_2  v^{-1} k} \,  \DE^{\otimes 2}\left[  \mathbf{1}_{L=k} \middle| \mathcal{F}_S\right]\nonumber\\
  & \leq e^{\Lambda_2 T_{\ell_0(\mathcal T,S)}}
  +   \sum_{k\geq 0} e^{\lambda_2 c_1 v^{-1} k } e^{-c_2 k}, \label{eq:boundBirknerGW}
\end{align}
where in the last inequality, $\ell_0$ is as in
Lemma \ref{lem:donotmeet},
$T_{\ell_0(\mathcal T,S)}(S)$
denotes the last time $k$ with $|S_k|\leq \ell_0(\mathcal T,S)$, and we used
\eqref{eq:donotmeet}.
Taking $\b$ small enough so that $\Lambda_2 c_1 <c_2 v$
ensures that
the right hand side of
\eqref{eq:boundBirknerGW} is finite, $GW\times \DP$-a.s. This
concludes the proof.
\end{proof}
\subsection{The canopy tree.} \label{sec:canopyProof}
We consider the walk $(\mathtt T,(S_k))$ on the canopy tree $\mathtt T$ defined in Section
\ref{sec:canopyTree}, with parameters $m=d>\lambda>1$. In this section, we prove the following:
\begin{theorem} \label{th:L2regionCanopy}
$\b_2>0$ for the polymer on $(\mathtt T,(S_k))$ with $m>\lambda>1$.
\end{theorem}

Even though the walk is in this case reversible and transient,
the existence of an $L^2$-region is not simply implied by Theorem
\ref{th:transImpL2} because, with $c_{xy}$ denoting the conductance of the edge
$(x,y)$,
the reversing
measure $\pi(x):=\sum_{x\sim y} c_{xy}= \lambda^{\ell-1} (\lambda+m)$ if
$x$ belongs to level $\ell$, is diverging when $\ell\to\infty$.

We briefly describe the situation. When $m>\lambda$, there is inside each finite $(m+1)$-regular sub-tree of the canopy tree a downwards drift which create traps for the walk, in the sense that the walk will spend a long time in finite subtrees at the bottom of the tree before exiting them forever by transience. Even though the walk will spend significant 
time in these finite sub-trees,
the branching structure ($m>1$) makes it hard for two independent copies
to meet frequently,
and this combined with transiences allow for an $L^2$-region to exist.

The rest of the section is devoted to the proof of Theorem
\ref{th:L2regionCanopy}. Section \ref{subsec-A}
develops some preliminary standard
(one dimensional) random walk estimates. Section
\ref{subsec-B} contains the actual proof.
\subsubsection{Random walk estimates}
\label{subsec-A}
Fix $\kappa>1$.
Consider the random walk $(X_k)$ on $\{0,1,\ldots,\ell\}$ with conductances $\gamma^i$ on the edge $(i,i+1)$, $i=0,\ldots,\ell-1$.
Let $\tau_i=\min\{t>0: X_t=i\}$. We write $\DP_i$ for the law of the random walk with $X_0=i$.
\begin{lemma}
There exists a constant $c>0$ such that for all $\ell > 0$ and $t>0$,
\begin{equation}
\label{eq-1}
\DP_\ell(\tau_0=t) \leq c \gamma^{-\ell}.
\end{equation}
\end{lemma}
\begin{proof}
By a conductance computation,
$q_\ell := \DP_\ell(\tau_0 < \tau_\ell) \leq c \gamma^{-\ell}$
for some constant $c$.
We use the pathwise decomposition
$\tau_0=\sum_{j=1}^{\mathcal G} \hat\tau^j+ \hat \tau_0$, where $\hat \tau^j$ are i.i.d.\ and have the law of $\tau_\ell$ under $\DP_\ell$ conditioned
on $\tau_\ell<\tau_0$,
${\mathcal G}$ is a geometric random variable of success parameter $q_\ell$, and
$\hat \tau_0$ has the law of $\tau_0$ under $\DP_0$
and conditioned on $\tau_0<\tau_\ell$, with independence of the
$\hat \tau^j$s, $\mathcal G$ and $\hat \tau_0$.
Then,
\begin{align*}
\DP_\ell\left(\sum\nolimits_{j=1}^{\mathcal G} \hat\tau^j = t\right) & = \sum_{k=0}^\infty \DP(\mathcal G = k) \DP_\ell\left(\sum\nolimits_{j=1}^{k} \hat\tau^j = t\right)\\
& \leq q_\ell \,\DE_\ell \left[ \sum_{k=0}^\infty  \mathbf{1}_{\sum\nolimits_{j=1}^{k} \hat\tau^j = t}\right]
 \leq c\gamma^{-\ell},
\end{align*}
where in the last inequality, we have used that the quantity inside the expectation on the second line is almost-surely bounded by $1$ since $k\to \sum_{j=1}^k \hat \tau^j$ visits at most once $t$ (the $\hat \tau^j$'s are positive).
We finally obtain that
\[\DP_\ell(\tau_0=t) = \sum_{s=0}^t \DP_\ell\left(\sum\nolimits_{j=1}^{\mathcal G} \hat\tau^j = s\right) \DP_\ell(\hat \tau_0 = t-s) \leq c\gamma^{-\ell}.\]
\end{proof}

\subsubsection{Proof of Theorem \ref{th:L2regionCanopy}}
\label{subsec-B}
Equipped with the estimates in Section \ref{subsec-A},
we can now proceed with the
\begin{proof}[Proof of Theorem
\ref{th:L2regionCanopy}]
We will prove that condition \eqref{eq:Khas} is verified, that is
\begin{equation} \label{eq:supCanopKhas}
\sup_{x \in\mathtt T} \sum_{t\geq 0} \sum_{y\in \mathcal T} p_t(x,y)^2 < \infty.
\end{equation}
For $x,y\in\mathtt T$, recall the definition of the Green function:
\[G(x,y) = \sum_{t\geq 0} p_t(x,y) = \DE_x\left[ \sum_{t\geq 0} \mathbf{1}_{S_t=y}\right].\]
For all $\ell\geq 0$, we label by $\ell\in \mathtt T$ the $\ell$-th
vertex on the left-most ray of $\mathtt T$,
see figure \ref{fig:canopy}.
We first observe that:
\begin{equation} \label{eq:bddGreenFncCanopy}
\sup_{\ell \geq 0} G(\ell,\ell) < \infty.
\end{equation}
Indeed, the number of visits of $(S)$ to $\ell$ is, by transience, stochastically dominated by the number of visits to site $\ell$ of biased random walk
on $\mathbb Z$, with jump probability to the right equal to $\lambda/(\lambda+1)$ and geometric holding times,
started at $\ell$. The law of the latter is independent of
$\ell$, and by transience, its Green function is bounded. This proves
\eqref{eq:bddGreenFncCanopy}.

We next show that:
\begin{equation}
\sup_{\ell_0 \geq 0} \sum_{t\geq 0} \sum_{y\in \mathtt T} p_t(\ell_0,y)^2 <\infty.
\end{equation}
Note that by the
symmetry properties of the canopy tree, this will directly entail \eqref{eq:supCanopKhas}. For all $\ell\geq 0$, let $\mathtt T^{(\ell)}$ denote the $(m+1)$-regular tree of height $\ell$ rooted at $\ell\in\mathtt T$ of degree m,
see Figure \ref{fig:canopy}. We decompose:
\begin{align*}
&\sum_t \sum_{y\in \mathtt T} p_t(\ell_0,y)^2\\
 & = \sum_t \sum_{y\in \mathtt T^{(\ell_0)}} p_t(\ell_0,y)^2 +
 \sum_t \sum_{\ell>\ell_0} \sum_{y\in \mathtt T^{(\ell)}\setminus \mathtt T^{(\ell-1)}} p_t(\ell_0,y)^2\\
& =: A_{\ell_0} + B_{\ell_0}.
\end{align*}
We first deal with $A_{\ell_0}$. For $\ell\geq 0$, define $\bar p^{(\ell)}_{t}$ to be the transition probability of $S$ after all vertices at the same level
in $\mathtt T^{(\ell)}$ have been glued together and the
corresponding conductances have been summed.
Label by  $w\in[0,\ell]$ the node of depth $w$ in the glued version of
$\mathtt T^{(\ell)}$, so that $\bar p^{(\ell)}_{t}(w,w')$ stands for the
probability for $(S_k)$ to go from depth $w$ to depth $w'$ in $\mathtt T^{(\ell)}$ after $t$ steps.

By symmetry, we observe that $p_t(\ell,y)=m^{-w}\bar p^{(\ell)}_{t}(0,w)$ for all $y\in \mathtt T^{(\ell)}$ of depth $w$. Moreover, by reversibilty,
 \begin{equation} \label{eq:revpbar}
 \bar p^{(\ell)}_{t-s}(0,w) = \left(\frac{m}{\lambda}\right)^{w} \bar p^{(\ell)}_{t-s}(w,0).
 \end{equation} Therefore,
\begin{align}
 \sum_{t\geq 0} \sum_{y\in \mathtt T^{(\ell_0)}} p_t(\ell_0,y)^2 & = \sum_{t\geq 0} \sum_{w=0}^{\ell_0} m^w m^{-2w} \bar p^{(\ell_0)}_{t}(0,w)^2 \label{eq:symmetryCanopy}\\
 & \leq \sum_{w=0}^\infty \lambda^{-w} \sum_{t\geq 0} \bar p^{(\ell_0)}_{t}(w,0).\nonumber
\end{align}
By the Markov property, the
identity $\bar p^{(\ell)}_{t}(0,0) = \bar p_t(\ell,\ell)$
and \eqref{eq:bddGreenFncCanopy}, we have that
\begin{equation} \label{eq:GreenMarkov}
\sum_{t\geq 0} \bar p^{(\ell)}_{t}(w,0) \leq \sum_{t\geq 0} \bar p^{(\ell)}_{t}(0,0) \leq \sup_{\ell \geq 0} G(\ell,\ell) < \infty,
\end{equation}
and it follows from \eqref{eq:symmetryCanopy}
that $\sup_{\ell_0} A_{\ell_0} < \infty$ since $\lambda > 1$.

We turn to $B_{\ell_0}$. By the Markov property, we have that
\[B_{\ell_0} = \sum_t \sum_{\ell> \ell_0}
\sum_{y\in \mathtt T^{(\ell)}\setminus \mathtt T^{(\ell-1)}} \left(\sum_{s=0}^t \DP_{\ell_0}(\tau_\ell = s) p_{t-s}(\ell,y)\right)^2\]
where $\tau_\ell$ denotes the first hitting time of $\ell \in \mathtt T$. By symmetry, as in \eqref{eq:symmetryCanopy},
\begin{align*}
B_{\ell_0} & \leq \sum_{t\geq 0} \sum_{\ell > \ell_0} \sum_{w=1}^{\ell} m^{-w} \left(\sum_{s=0}^t \DP_{\ell_0}(\tau_\ell = s) \,\bar p^{(\ell)}_{t-s}(0,w) \right)^2\\
& =: B_{\ell_0}^{(1)} + B_{\ell_0}^{(2)},
\end{align*}
where in $B_{\ell_0}^{(1)}$ the sum in $w$ is restricted to $w\leq \ell -\ell_{0}$.
By \eqref{eq:revpbar} and since $\sum_{s=0}^t \DP_{\ell_0}(\tau_\ell = s) \,\bar p^{(\ell)}_{t-s}(0,w)\leq 1$, we have that
\begin{align*}
B_{\ell_0}^{(2)} \leq \sum_{t\geq 0} \sum_{\ell > \ell_0} \sum_{w=\ell-\ell_0+1}^{\ell} \lambda^{-w} \left(\sum_{s=0}^t \DP_{\ell_0}(\tau_\ell = s) \,\bar p^{(\ell)}_{t-s}(w,0) \right),
\end{align*}
so by summing first over $t\geq 0$ and using \eqref{eq:GreenMarkov}, we find that
\[
\sup_{\ell_0 \geq 0} B_{\ell_0}^{(2)} \leq \sup_{\ell \geq 0} G(\ell,\ell) \sum_{\ell > \ell_0} \lambda^{-(\ell-\ell_0+1)} < \infty.
\]
Identity \eqref{eq:revpbar} further gives that
\[
B_{\ell_0}^{(1)} = \sum_{t\geq 0} \sum_{\ell > \ell_0} \sum_{w=1}^{\ell-\ell_0} m^{-w} \left(\frac{m}{\lambda}\right)^{2w} \left(\sum_{s=0}^t \DP_{\ell_0}(\tau_\ell = s) \,\bar p^{(\ell)}_{t-s}(w,0) \right)^2,
\]
and we have by estimates \eqref{eq-1} with conductance $m/\lambda$, and  \eqref{eq:GreenMarkov}, letting $k_\ell = c \left( m / \lambda \right)^\ell$,
\begin{equation*} \label{eq:canopy_kl_bound}
\left(\sum_{s=0}^t \DP_{\ell_0}(\tau_\ell = s) \,\bar p^{(\ell)}_{t-s}(w,0) \right) \leq k_{\ell-\ell_0}^{-1} \sum_{t\geq 0} p^{(\ell)}_{t}(w,0) \leq C\, k_{\ell-\ell_0}^{-1},
\end{equation*}
for some finite $C$, so that
\begin{align*}
B_{\ell_0}^{(1)} \leq C \sum_{\ell>\ell_0} k_{\ell-\ell_0}^{-1} \sum_{w=1}^{\ell-\ell_0}  \left(\frac{m}{\lambda^2}\right)^{w} \left(\sum_{t\geq 0} \sum_{s=0}^t \DP_{\ell_{0}}(\tau_\ell = s) \,\bar p^{(\ell)}_{t-s}(w,0) \right).
\end{align*}
The sum inside the above parenthesis is again uniformly bounded from above. Moreover, there is some $C<\infty$ such that:
\[\sum_{w=1}^\ell  \left(\frac{m}{\lambda^2}\right)^{w} \leq \begin{cases}
 C & \text{if } m<\lambda^2,\\
 \ell & \text{if } m = \lambda^2,\\
 C\left({m}/{\lambda^2}\right)^\ell & \text{if } m>\lambda^2.
\end{cases}
\]
In any case, $\sum_{\ell > 0} k_\ell^{-1} \sum_{w=1}^\ell  \left({m}/{\lambda^2}\right)^{w} < \infty$, so putting things together we obtain that $\sup_{\ell_0} B_{\ell_0} <\infty$, which concludes the proof.
\end{proof}

\subsection{Presence of an $L^2$-region for a recurrent walk} \label{sec:L2regionRec}
Let $\mathcal T_2$ be the infinite binary tree and denote its root by $o$. We will 
define a walk on $T_2$ which, at each step, goes 
down with probability 1/2 from a vertex to one of its children, 
until a clock rings and brings the walk back to the root. 
If the distribution of the clock has a sufficiently heavy tail, 
the number of intersections of two independent walks 
will be small enough to allow for $\b_2>0$. 

Turning to the actual construction,
consider the graph product $G_0 = \mathcal T_2 \times \mathbb Z^2$ where $\mathbb Z^2$ stands for the two-dimensional lattice. Let $G=(V,E)$ be the subgraph of $G_0$ where $V=\{(v,x)\in G_0 : d(o,v) = \Vert x \Vert_1\}$ and $E$ contains the edges of $G_0$ for which both ends are in $V$.
Let $q(x,y)=1/4$ if $x,y\in \mathbb Z^2$ satisfy $|x-y|=1$.  Define the walk $S_n=(T_n,X_n)$ on $G$ by the transition probabilities 
\[
p((v,x),(w,y))= \begin{cases}
\frac{1}{2} q(x,y) 
& \text{ if } w \text{ is a child of } v \text{ and } y\neq 0,\\
q(x,y)
& \text{ if } y=0 \text{ and } w=o,\\
 0 & \text{ otherwise,}
\end{cases} 
\] for all edges $((v,x),(w,y))\in E$.
Defined as such, $(X_n)$ is the SRW on $\mathbb Z^2$ and $T_n$ is a walk on the binary tree such that $T_n$ jumps to one of his children with probability $1/2$ when $X_{n+1}\neq 0$ and jumps back to the root when $X_{n+1}=0$.

\begin{theorem} The walk $(S_n)$ is recurrent and satisfies $\b_2>0$ for the associated polymer.
\end{theorem}
\begin{proof}
Note that $(X_n)$ is the SRW on $\mathbb Z^2$. Since $(X_n)$ visits infinitely many often $0$, and since $T_n=o$ whenever this happens, $(S_n)$ is recurrent. 

We now check that the condition in \eqref{eq:Khas} is verified. For two independent copies $S=(X,T),S'=(X',T')$ of $S$, let $\tau_0=0$ and define recursively $\tau_{n+1} = \inf\{k> \tau_n : X_k=X_k'=0\}$, with the convention that the infimum over an empty set is equal to infinity. Further let $K=K(X,X')=\sum_{k\geq 0} \mathbf{1}_{X_k=X'_k=0}$. 

For any $(v,x)\in V$, we have
\begin{align*}
\DE_{(v,x)}^{\otimes 2}\left[\sum_{n=1}^\infty \mathbf{1}_{S_n=S_n'}\right]&  \leq \DE_{(v,x)}^{\otimes 2}\left[\sum_{n=1}^\infty \mathbf{1}_{T_n=T_n'}\right]\\
& = \DE_{(v,x)}^{\otimes 2}\left[\sum_{k=0}^K\DE_{(v,x)}^{\otimes 2}\left[\sum_{n=\tau_k}^{\tau_{k+1}-1} \mathbf{1}_{T_n=T_n'}\middle| X,X'\right]\right].
\end{align*}
Let $\tilde{T},\tilde{T}'$ be independent copies of the random walk on $\mathcal{T}_2$ which at each step, goes down to one of its children with probability $1/2$. By Markov's property, we have
\[
\DE_{(v,x)}^{\otimes 2}\left[\sum_{n=\tau_k}^{\tau_{k+1}-1} \mathbf{1}_{T_n=T_n'}\middle| X,X'\right] \leq \DE_{o}^{\otimes 2}\left[\sum_{n=0}^{\infty} \mathbf{1}_{\tilde{T}_n=\tilde{T}_n'}\right] = 1 + \DE[\mathcal G],
\]
where $\mathcal G$ is a geometric random variable of parameter $1/2$. Therefore,
\begin{align*}
\DE_{(v,x)}^{\otimes 2}\left[\sum_{n=1}^\infty \mathbf{1}_{S_n=S_n'}\right]&  \leq (1 + \DE[\mathcal G]) \DE_{(v,x)}^{\otimes 2}\left[K\right],
\end{align*}
where $\DE_{(v,x)}^{\otimes 2}\left[K\right] = \sum_{n\geq 0} \DP_x(X_n=0)^2 \leq \sum_{n\geq 0} cn^{-2}<\infty$ by the local limit theorem for the simple random walk, uniformly in $(v,x)\in V$. Hence, by Khasminskii's lemma,  $\b_2>0$ for the polymer associated with $(S_n)$.
\end{proof}

\section{Conclusions and open problems}
\label{sec-open}
We have presented some elements of a theory of polymers on general graphs.
Our study leaves several important open questions. The comments below address
some of these.

\begin{enumerate}
  \item It is natural to wonder whether, in a conductance model,
    $\beta_c$ is monotone with respect to adding edges. Maybe counter-intuitively, the answer is no.  The following is a counter example.
Consider the SRW on $\mathbb Z_+^3$, which  satisfies
$\b_c>0$. This corresponds to a conductance model on $\mathbb Z_+^3$
with all edges having conductance equal to $1$. Now increase the
conductance between
$(i,j,k)\in \Z^3_+$
and
$(i+1,j,k)$ to $e^{e^i}$.
(This corresponds to adding edges between $i$ and $i+1$,
more and more as $i\to\infty$.)
Then, by Borel-Cantelli,
the random walk eventually goes ballistically on a path
$(i,a,b)\to (i+1,a,b)\to (i+2,a,b)\ldots$, with some random $a,b$,
and in fact the probability that this
did not happen before time $n$ decays super exponentially in $n$.
Now a repeat of the proof of Proposition
\ref{prop:counterExUI} leads to the conclusion that $\b_c=0$.
\item  We showed in Section \ref{sec:L2regionRec} that there is a pair $(G,S)$ where $S$ is recurrent but the polymer has $\b_2>0$. Is there a \textit{reversible} recurrent
pair $(G,S)$ 
 with $\b_c>0$ or
 $\b_2>0$?
  Theorem \ref{th:noL2rec} shows that such pair with $\b_2>0$
 must necessarily satisfy that $\inf_x \pi(x)=0$. We note that recurrent graphs with a.s.\ finitely many collisions of two copies of SRW
were constructed in \cite{KP04}, however in those examples $\b_2=0$.

\item Is $\b_c>0$ for SRW on the supercritical percolation cluster
  on $\Z^d$, $d\geq 3$? Theorem \ref{th:noL2perco} shows that $\b_2=0$ in that
  setup. For $d=2$, $\b_c=0$ as follows either Remarks \ref{rem-CV} or
  \ref{rem-Lacoin}.
\item  In Section \ref{sec-5}, we showed that for the $\lambda$-biased
  random walk on the Galton-Watson tree, $\b_c>0$ in the transient case
  (with $p_0=0$)
  and  $\b_c=0$ in the positive
  recurrent case. We left open the null recurrent case
  ($m=\lambda$). An application of \cite{peres2008central}
  shows that $\b_2=0$, even in the case $p_1=0$. Unfortunately,
  it seems that the CLT in \cite{peres2008central} rules out
  the applicability of Theorem \ref{th:Birkner}.
\end{enumerate}

\section*{Acknowledgement}
This project has received funding from the European Research Council (ERC) under the European
Union Horizon 2020 research and innovation program (grant agreement No. 692452). 

\bibliographystyle{plain}
{\footnotesize \bibliography{polymeres-bib}

\begin{thebibliography}{10}

\bibitem{AiWa06Canopy}
Michael Aizenman and Simone Warzel.
\newblock The canopy graph and level statistics for random operators on trees.
\newblock {\em Mathematical Physics, Analysis and Geometry}, 9(4):291--333,
  2006.

\bibitem{Barlow}
Martin~T. Barlow.
\newblock Random walks and diffusions on fractals.
\newblock In {\em Proceedings of the {I}nternational {C}ongress of
  {M}athematicians, {V}ol. {I}, {II} ({K}yoto, 1990)}, pages 1025--1035. Math.
  Soc. Japan, Tokyo, 1991.

\bibitem{barlowperc}
Martin~T. Barlow.
\newblock Random walks on supercritical percolation clusters.
\newblock {\em Ann. Probab.}, 32(4):3024--3084, 2004.

\bibitem{BB}
Martin~T. Barlow and Richard~F. Bass.
\newblock Random walks on graphical {S}ierpinski carpets.
\newblock In {\em Random walks and discrete potential theory ({C}ortona,
  1997)}, Sympos. Math., XXXIX, pages 26--55. Cambridge Univ. Press, Cambridge,
  1999.

\bibitem{BatesChatterjee}
Erik Bates and Sourav Chatterjee.
\newblock The endpoint distribution of directed polymers.
\newblock {\em Ann. Probab.}, 48(2):817--871, 03 2020.

\bibitem{BergerToninelli}
Quentin Berger and Fabio~Lucio Toninelli.
\newblock On the critical point of the random walk pinning model in dimension
  {$d=3$}.
\newblock {\em Electron. J. Probab.}, 15:no. 21, 654--683, 2010.

\bibitem{Bir04}
Matthias Birkner.
\newblock A condition for weak disorder for directed polymers in random
  environment.
\newblock {\em Electron. Comm. Probab.}, 9:22--25 (electronic), 2004.

\bibitem{birkner2011collision}
Matthias Birkner, Andreas Greven, Frank den Hollander, et~al.
\newblock Collision local time of transient random walks and intermediate
  phases in interacting stochastic systems.
\newblock {\em Electronic Journal of Probability}, 16:552--586, 2011.

\bibitem{BirknerSun}
Matthias Birkner and Rongfeng Sun.
\newblock Annealed vs quenched critical points for a random walk pinning model.
\newblock {\em Ann. Inst. Henri Poincar{\'e} Probab. Stat.}, 46(2):414--441,
  2010.

\bibitem{BirknerSun11}
Matthias Birkner and Rongfeng Sun.
\newblock Disorder relevance for the random walk pinning model in dimension 3.
\newblock {\em Ann. Inst. Henri Poincar{\'e} Probab. Stat.}, 47(1):259--293,
  2011.

\bibitem{Bolth89}
Erwin Bolthausen.
\newblock A note on the diffusion of directed polymers in a random environment.
\newblock {\em Comm. Math. Phys.}, 123(4):529--534, 1989.

\bibitem{BrunetE2000Pdot}
E~Brunet and B~Derrida.
\newblock Probability distribution of the free energy of a directed polymer in
  a random medium.
\newblock {\em Physical review. E, Statistical physics, plasmas, fluids, and
  related interdisciplinary topics}, 61(6 Pt B):6789--6801, 2000.

\bibitem{EBuffet1993Dpot}
E~Buffet.
\newblock Directed polymers on trees: a martingale approach.
\newblock {\em Journal of Physics A: Mathematical and General},
  26(8):1823--1834, 1993.

\bibitem{CaGuHuMe04}
Philippe Carmona, Francesco Guerra, Yueyun Hu, and Olivier Mejane.
\newblock Strong disorder for a certain class of directed polymers in a random
  environment.
\newblock {\em Journal of Theoretical Probability}, 19, 04 2004.

\bibitem{CaHu02}
Philippe Carmona and Yueyun Hu.
\newblock On the partition function of a directed polymer in a {G}aussian
  random environment.
\newblock {\em Probab. Theory Related Fields}, 124(3):431--457, 2002.

\bibitem{CStFlour}
Francis Comets.
\newblock {\em {Directed polymers in random environments. \'Ecole d'\'Et\'e de
  Probabilit\'es de Saint-Flour XLVI -- 2016.}}
\newblock Cham: Springer, 2017.

\bibitem{CL16}
Francis Comets and Quansheng Liu.
\newblock Rate of convergence for polymers in a weak disorder.
\newblock {\em Journal of Mathematical Analysis and Applications}, 455(1):312
  -- 335, 2017.

\bibitem{CoFlRa19}
Francis Comets, Gregorio Moreno, and Alejandro~F. Ram\'{\i}rez.
\newblock Random polymers on the complete graph.
\newblock {\em Bernoulli}, 25(1):683--711, 2019.

\bibitem{CSY03}
Francis Comets, Tokuzo Shiga, and Nobuo Yoshida.
\newblock Directed polymers in a random environment: path localization and
  strong disorder.
\newblock {\em Bernoulli}, 9(4):705--723, 2003.

\bibitem{CV06}
Francis Comets and Vincent Vargas.
\newblock Majorizing multiplicative cascades for directed polymers in random
  media.
\newblock {\em ALEA Lat. Am. J. Probab. Math. Stat.}, 2:267--277, 2006.

\bibitem{CY06}
Francis Comets and Nobuo Yoshida.
\newblock Directed polymers in random environment are diffusive at weak
  disorder.
\newblock {\em The Annals of Probability}, pages 1746--1770, 2006.

\bibitem{CYBMPO2}
Francis Comets and Nobuo Yoshida.
\newblock Localization transition for polymers in {P}oissonian medium.
\newblock {\em Comm. Math. Phys.}, 323(1):417--447, 2013.

\bibitem{CoNa20}
Cl\'{e}ment Cosco and Shuta Nakajima.
\newblock Gaussian fluctuations for the directed polymer partition function for
  $d\geq 3$ and in the whole $l^2$-region.
\newblock {\em To appear in Ann. IHP. arXiv:1903.00997}, 2020.

\bibitem{CoNaNa20}
Cl\'{e}ment Cosco, Shuta Nakajima, and Makoto Nakashima.
\newblock Law of large numbers and fluctuations in the sub-critical and {$L^2$}
  regions for {SHE} and {KPZ} equation in dimension $d\geq 3$.
\newblock 2020.

\bibitem{dembo2002large}
Amir Dembo, Nina Gantert, Yuval Peres, and Ofer Zeitouni.
\newblock Large deviations for random walks on galton--watson trees: averaging
  and uncertainty.
\newblock {\em Probability theory and related fields}, 122(2):241--288, 2002.

\bibitem{PAMonGW_HoKoSa20}
Frank den Hollander, Wolfgang K\"{o}nig, and Renato~S. dos Santos.
\newblock The parabolic anderson model on a galton-watson tree, 2020.

\bibitem{DerridaSpohn88}
Bernard Derrida and Herbert Spohn.
\newblock Polymers on disordered trees, spin glasses, and traveling waves.
\newblock {\em J. Statist. Phys.}, 51(5-6):817--840, 1988.
\newblock New directions in statistical mechanics (Santa Barbara, CA, 1987).

\bibitem{DuGuRyZe20}
Alexander Dunlap, Yu~Gu, Lenya Ryzhik, and Ofer Zeitouni.
\newblock Fluctuations of the solutions to the {KPZ} equation in dimensions
  three and higher.
\newblock {\em Probab. Theory Relat. Fields}, 2020.

\bibitem{EckmannJ.1989TlLe}
J.~Eckmann and C.~Wayne.
\newblock The largest liapunov exponent for random matrices and directed
  polymers in a random environment.
\newblock {\em Communications in Mathematical Physics}, 121(1):147--175, 1989.

\bibitem{GKZ93}
G.~R. Grimmett, H.~Kesten, and Y.~Zhang.
\newblock Random walk on the infinite cluster of the percolation model.
\newblock {\em Probab. Theory Related Fields}, 96(1):33--44, 1993.

\bibitem{grimmett}
Geoffrey Grimmett.
\newblock {\em Percolation}, volume 321 of {\em Grundlehren der Mathematischen
  Wissenschaften}.
\newblock Springer-Verlag, Berlin, second edition, 1999.

\bibitem{GuRyZe18}
Yu~Gu, Lenya Ryzhik, and Ofer Zeitouni.
\newblock The {E}dwards--{W}ilkinson limit of the random heat equation in
  dimensions three and higher.
\newblock {\em Communications in Mathematical Physics}, 363(2):351--388, 2018.

\bibitem{HamblyKumagaiSierpinski}
B.~M Hambly and T~Kumagai.
\newblock Asymptotics for the spectral and walk dimension as fractals approach
  euclidean space.
\newblock {\em Fractals}, 10(4):403--412, 2002.

\bibitem{Harris}
T.~E. Harris.
\newblock A lower bound for the critical probability in a certain percolation
  process.
\newblock {\em Proc. Cambridge Philos. Soc.}, 56:13--20, 1960.

\bibitem{HuHe85}
David~A. Huse and Christopher~L. Henley.
\newblock Pinning and roughening of domain walls in {I}sing systems due to
  random impurities.
\newblock {\em Physical Review Letters}, 54(25):2708, 1985.

\bibitem{ImbrieSpencer88}
John Imbrie and Thomas Spencer.
\newblock Diffusion of directed polymers in a random environment.
\newblock {\em J. Statist. Phys.}, 52(3-4):609--626, 1988.

\bibitem{jones}
Owen~Dafydd Jones.
\newblock Transition probabilities for the simple random walk on the
  {S}ierpi\'{n}ski graph.
\newblock {\em Stochastic Process. Appl.}, 61(1):45--69, 1996.

\bibitem{KestenStigum66}
H.~Kesten and B.~P. Stigum.
\newblock A limit theorem for multidimensional {G}alton-{W}atson processes.
\newblock {\em Ann. Math. Statist.}, 37:1211--1223, 1966.

\bibitem{KP04}
Manjunath Krishnapur and Yuval Peres.
\newblock Recurrent graphs where two independent random walks collide finitely
  often.
\newblock {\em Electron. Comm. Probab.}, 9:72--81, 2004.

\bibitem{KumagaiTakashi2014RWoD}
Takashi Kumagai.
\newblock {\em Random Walks on Disordered Media and their Scaling Limits:
  École d'Été de Probabilités de Saint-Flour XL - 2010}, volume 2101 of
  {\em Lecture Notes in Mathematics}.
\newblock Springer International Publishing, Cham, 2014 edition, 2014.

\bibitem{Lacoin10}
Hubert Lacoin.
\newblock New bounds for the free energy of directed polymers in dimension
  {$1+1$} and {$1+2$}.
\newblock {\em Comm. Math. Phys.}, 294(2):471--503, 2010.

\bibitem{LiuWat}
Quansheng Liu and Fr{{\'e}}d{{\'e}}rique Watbled.
\newblock Exponential inequalities for martingales and asymptotic properties of
  the free energy of directed polymers in a random environment.
\newblock {\em Stochastic Process. Appl.}, 119(10):3101--3132, 2009.

\bibitem{LyZy20}
Dimitris Lygkonis and Nikos Zygouras.
\newblock Edwards-wilkinson fluctuations for the directed polymer in the full
  $l^2$-regime for dimensions $d \geq 3$.
\newblock 2020.

\bibitem{lyons1990random}
Russell Lyons.
\newblock Random walks and percolation on trees.
\newblock {\em The annals of Probability}, pages 931--958, 1990.

\bibitem{LyonsPemantlePeresBiased}
Russell Lyons, Robin Pemantle, and Yuval Peres.
\newblock Biased random walks on galton-watson trees.
\newblock {\em Probability Theory and Related Fields}, 106, 10 1996.

\bibitem{LPbook}
Russell Lyons and Yuval Peres.
\newblock {\em Probability on trees and networks}, volume~42 of {\em Cambridge
  Series in Statistical and Probabilistic Mathematics}.
\newblock Cambridge University Press, New York, 2016.

\bibitem{MaUn18}
Jacques Magnen and J{\'e}r{\'e}mie Unterberger.
\newblock The scaling limit of the {KPZ} equation in space dimension 3 and
  higher.
\newblock {\em Journal of Statistical Physics}, 171(4):543--598, May 2018.

\bibitem{mathieu}
Pierre Mathieu.
\newblock Carne-{V}aropoulos bounds for centered random walks.
\newblock {\em Ann. Probab.}, 34(3):987--1011, 2006.

\bibitem{peres2008central}
Yuval Peres and Ofer Zeitouni.
\newblock A central limit theorem for biased random walks on galton--watson
  trees.
\newblock {\em Probability Theory and Related Fields}, 140(3-4):595--629, 2008.

\bibitem{seroussi2018spectral}
Inbar Seroussi and Nir Sochen.
\newblock Spectral analysis of a non-equilibrium stochastic dynamics on a
  general network.
\newblock {\em Scientific reports}, 8(1):1--10, 2018.

\bibitem{Shi-StFl}
Zhan Shi.
\newblock {\em Branching random walks}, volume 2151 of {\em Lecture Notes in
  Mathematics}.
\newblock Springer, Cham, 2015.
\newblock Lecture notes from the 42nd Probability Summer School held in Saint
  Flour, 2012.

\bibitem{Sznitman}
Alain-Sol Sznitman.
\newblock {\em Brownian motion, obstacles and random media}.
\newblock Springer Monographs in Mathematics. Springer-Verlag, Berlin, 1998.

\bibitem{Viveros20}
Roberto Viveros.
\newblock Directed polymer for very heavy tailed random walks, 2020.

\bibitem{W00}
Wolfgang Woess.
\newblock {\em Random walks on infinite graphs and groups}, volume 138 of {\em
  Cambridge Tracts in Mathematics}.
\newblock Cambridge University Press, Cambridge, 2000.

\end{thebibliography}
}

\end{document}